\documentclass[12pt,a4paper]{article}

\usepackage{latexsym,amssymb}
\usepackage{amsmath}
\usepackage{amsthm}
\usepackage[cp1250]{inputenc}
\usepackage[T1]{fontenc}
\usepackage[english]{babel}
\usepackage[a4paper, left=3.2cm, right=2.8cm, top=2.3cm, bottom=3.4cm, headsep=1.2cm]{geometry}
\usepackage{enumerate}

\newtheorem{theorem}{Theorem}[section]
\newtheorem{lemma}[theorem]{Lemma}
\newtheorem{proposition}[theorem]{Proposition}

\theoremstyle{definition}
\newtheorem{definition}[theorem]{Definition}

\theoremstyle{remark}
\newtheorem{remark}[theorem]{Remark}

\def\R{\mathbb{R}}
\def\N{\mathbb{N}}
\def\Gr{\mathrm{Gr}}
\def\d{\mathrm{\,d}}
\def\lma{\lambda}
\def\part{\partial}

\begin{document}
\title{\textbf{{Stationary solutions and connecting orbits for  $p$-Laplace equation}}}
\author{Aleksander \'{C}wiszewski, Mateusz Maciejewski\\
\emph{Faculty of Mathematics and Computer Science} \\
\emph{Nicolaus Copernicus University}\\
\emph{ul. Chopina 12/18, 87-100 Toru\'{n}, Poland}}

\maketitle

\begin{abstract}
We deal with one dimensional $p$-Laplace equation of the form
$$
u_t = (|u_x|^{p-2} u_x )_x + f(x,u), \ x\in (0,l), \ t>0,
$$
under Dirichlet boundary condition, where $p>2$ and $f\colon [0,l]\times \mathbb{R}\to \mathbb{R}$ is a continuous function with $f(x,0)=0$. We will prove that if there is at least one eigenvalue of the $p$-Laplace operator between $\lim_{u\to 0} f(x,u)/|u|^{p-2}u$ and $\lim_{|u|\to +\infty} f(x,u)/|u|^{p-2}u$, then there exists a nontrivial stationary solution. Moreover we show the existence of a connecting orbit between stationary solutions. The results are obtained by use of Conley type homotopy index and continuation
along $p$ techniques. \end{abstract}

\section{Introduction}
We shall study a nonlinear $p$-Laplace equation
\begin{equation} \label{07042015-1828}
\left\{ \begin{array}{l}
u_t (t,x) = \left(\left| u_x (t,x)\right|^{p-2}u_x (t,x)\right)_x + f(x,u(t,x)), \ x\in (0,l), \ t\in\R,\\
u(t,0)=u(t,l)=0,\ t\in\R,
\end{array}
\right.
\end{equation}
with $p>2$, $l>0$ a continuous $f\colon [0,l]\times \R\to \R$, which is locally Lipschitz with respect to the second variable, i.e.
\begin{eqnarray}\label{0128-29042015}
\mbox{ for any $R>0$ there is $L>0$ such that } |f(x,u)-f(x,v)|\leq L|u-v|\\
\mbox{ for all } x\in [0,l],\, u,v \in [-R,R]. \nonumber \nonumber
\end{eqnarray}
The stationary version of (1), i.e. the elliptic problem
\begin{equation}\label{1348-21082015}
\left\{
\begin{array}{l}
-(|u'(x)|^{p-2}u'(x))'=f(x,u(x)), \ x\in (0,l),\\
\ \ u(0)=u(l)=0,
\end{array}
\right.
\end{equation}
is subject of extensive studies by many authors -- see earlier papers \cite{Manasevich}, \cite{Manasevich-et-al-1}, \cite{Drabek}, \cite{Man-Njo-Za} as well as more recent examples \cite{Mon-Mon-Papa}, \cite{cw-kr}, \cite{cw-mac}, \cite{Precup},
\cite{Lan-Yang-Yang} or \cite{Maciejewski2015}. Usually topological degree/index techniques or variational approach are applied.
Here we use a dynamical system approach based on the Conley type index from \cite{rybakowski}  and \cite{rybakowski-TAMS}
and Rybakowski's techniques from \cite{Rybakowski-JDE1984}. To compute the Conley indices, inspired by \cite{Manasevich}, we use deformation along $p$. We shall prove the following existence criterion.
\begin{theorem}\label{0914-29032015}
Suppose that $f\colon [0,l]\times\R\to\R$ is locally Lipschitz with respect to the second variable, $f(x,0)=0$ for $x\in (0,l)$ and
\begin{equation}\label{1219-21082015}
\lim_{u\to 0} \frac{f(x,u)}{|u|^{p-2}u}=f'_0 (x)
\end{equation}
and
\begin{equation}\label{1220-21082015}
\lim_{|u|\to +\infty} \frac{f(x,u)}{|u|^{p-2}u}=f'_\infty (x)
\end{equation}
for some $f'_0, f'_\infty\in C([0,l])$ uniformly with respect to $x\in [0,l]$. Suppose there are $k_0, k_\infty \in  \N$ such that $\lma_{k_0}^{(p)} \leq f'_0(x) \leq \lma_{k_0+1}^{(p)}$ and $\lma_{k_\infty}^{(p)} \leq f'_\infty (x) \leq \lma_{k_\infty+1}^{(p)}$, for all $x\in (0,l)$, with the strict inequalities on set of positive measure.
If $k_0\neq k_\infty$, then there exists a nontrivial  solution $\bar u\in C^{1}([0,l])$ of {\em (\ref{1348-21082015})}.\\
\indent Moreover, there exists a connecting orbit between $\bar u$ and $0$, i.e. a full solution $u$ of {\em (\ref{07042015-1828})} such that  either $u(t_n,\cdot) \to \bar u$ for some  $t_n\to +\infty$ and $u(t,\cdot) \to 0$ as $t\to -\infty$
or $u(t_n,\cdot) \to \bar u$ for some $t_n \to -\infty$ and $u(t,\cdot) \to 0$ as $t\to +\infty$ {\em (}with respect to the max norm of the space $C(0,l)${\em)}.
\end{theorem}
\noindent Here recall that $\lma\in\R$, for which the problem
$$\left\{
\begin{array}{l}
-(|u'(x)|^{p-2}u'(x))'=\lma |u(x)|^{p-2} u(x), \ x\in (0,l),\\
\ \ u(0)=u(l)=0,
\end{array}
\right.
$$
has nonzero solutions, make a sequence of positive numbers $\lma_n^{(p)}$, $n\geq 1$, such that $\lma_n^{(p)}\to +\infty$ and, for any $n\geq 1$, $\lma_n^{(q)}\to \lma_n^{(p)}$ whenever $q\to p$ (see \cite{Otani}). We also put $\lambda_0^{(p)}:=-\infty$.\\
\indent In this paper we consider a local semiflow (a sort of dynamical system) ${\mathbf \Phi^{(p,f)}}$ on the space ${\mathbf X}=C_0(0,l):=\{ u\in C(0,l)\mid u(0)=u(l)=0\}$ associated with the equation (\ref{07042015-1828}).  To find a stationary solution and related connecting trajectory we use the theory of irreducible sets due to Rybakowski \cite{rybakowski}, where we need to find Conley indices of the zero $K_0:=\{ 0\}$ and the set $K_\infty$ made by all full bounded trajectories of (\ref{07042015-1828}). The main difficulty lies in the fact that both the differential operator and continuous term are nonlinear.
In order to consider and compute Conley index (due to Rybakowski) for ${\mathbf \Phi}^{(p,f)}$ we need to study the existence, compactness and continuity properties of solutions. We shall also exploit the Lyapunov function for the problem and related regularity to find stationary solutions at the ends of full trajectories. What we gain by use of Conley index and what we could not obtain with topological degree techniques is that we show the existence under the condition $k_0\neq k_\infty$ while topological degree works in the case where $k_0$ and $k_\infty$ are of different parities. In addition, we have a full trajectory between two stationary solutions.\\
\indent The paper is organized as follows. In the rest of the section we give some notation and basic preliminaries on Conley index. Section 2 is devoted mainly to continuity and compactness issues for abstract evolution equations governed by perturbations of $m$-accretive operators and subdifferentials of convex functionals. In Section 3 we study the existence and regularity of solutions together with Lyapunov function theory. The continuity and compactness properties with respect to $p$ and $f$, which are crucial for computing Conley index via its continuation property, are explored in Section 4. Finally, we compute the Conley indices of $K_0$ and $K_\infty$ and prove Theorem \ref{0914-29032015} in Section 5.\\

\noindent {\bf Preliminaries}\\
Two pointed topological spaces $(X,x_0)$ and $(Y,y_0)$ are said to be {\em homotopically equivalent} or have {\em the same homotopy type}  if and only if there are maps $f\colon (X,x_0)\to (Y,y_0)$ and $g\colon (Y,y_0)\to (X,x_0)$ such that $f\circ g$ is homotopic via a mapping keeping $y_0$ fixed to the identity of $(Y,y_0)$
and $g\circ f$ is homotopic via a mapping keeping $x_0$ fixed to the identity of $(X,x_0)$. The homotopy class represented by a space $(X_0,x_0)$ is denoted by $[(X,x_0)]$. If $(X,A)$ is a pair of topological spaces with a nonempty and closed $A\subset X$, then $X/A$ denotes the quotient space, obtained by collapsing the subset $A$ to a point $[A]$.\\
\indent Consider $C_0(0,l):=\{ u\colon [0,l]\to \R \mid f \mbox{ is } continuous, \ u(0)=u(l)=0\}$ ($l>0$) with the usual norm $\|u\|_{\infty}:=\max_{[0,l]}|u|$. By $L^p(0,l)$ and $W^{k,p}(0,l)$ and $W_{0}^{1,p}(0,l)$ we denote the standard Lebesgue and Sobolev spaces on the interval $(0,l)$ and we put $H_0^{1}(0,l):=W_{0}^{1,2}(0,l)$. In the same way, by $L^p (0,T;X)$ and $W^{1,p}(0,T;X)$ we denote the spaces with values in a Banach space $X$.\\

\noindent {\bf Conley index due to Rybakowski}\\
Here we briefly present homotopy index theory from \cite{rybakowski} (see also \cite{rybakowski-TAMS}). Let $\Phi\colon D\to X$, where $D$ is an open subset of $[0,\infty)\times X$, be a local semiflow on a metric space $X$. Let $T_{\bar u}=\sup\{t>0\mid (t,\bar u)\in D\}$. A continuous function $u\colon J\to X$, where $J\subset \R$ is an interval, is called a {\em solution of $\Phi$} if and only if $u(t+s) = \Phi_t (u(s))$ for any $t\geq 0$ and $s\in J$ such that $t+s\in J$. If $u\colon [a,+\infty)\to X$, $a\in\R$,  is a solution of $\Phi$, then by the {\em $\omega$-limit of $u$} we mean the set
$$
\omega (u):=\left\{ \bar u\in X \mid \exists\, (t_n) \mbox{ in } [a,+\infty) \mbox{ such that } t_n\to +\infty \mbox{ and }\bar u
= \lim_{n\to\infty} u(t_n) \right\}.
$$
The {\em $\alpha$-limit} of a solution $u\colon (-\infty, a]\to X$ of $\Phi$, $a\in\R$, is defined by
$$
\alpha (u):=\left\{\bar u\in X \mid \exists\, (t_n) \mbox{ in } (-\infty,a] \mbox{ such that } t_n\to -\infty \mbox{ and } \bar u =\lim_{n\to\infty} u(t_n) \right\}.
$$
Let $N\subset X$. By the \emph{invariant part} $\mathrm{Inv}_\Phi (N)$ of $N$ we mean
$$
\mathrm{Inv}_\Phi (N) :=\{\bar u \in N \mid \exists \mbox{ a solution $u\colon \R\to N$ of $\Phi$ with }u(0)=\bar u \}.
$$
We shall say that $K\subset X$ is a {\em $\Phi$-invariant set} or {\em invariant with respect to $\Phi$} provided  $\mathrm{Inv}_\Phi (K)=K$.
A $\Phi$-invariant set $K\subset X$ is called an {\em isolated $\Phi$-invariant set} if and only if  there exists $N\subset X$ such that  $K=\mathrm{Inv}_\Phi (N) \subset \mathrm{int}\, N$. Such $N$ is called an {\em isolating neighborhood of $K$}.
The following concept of admissibility is crucial in Rybakowski's version of Conley theory on general metric spaces and enables us to construct a Conley type index without local compactness of $X$.
\begin{definition}\label{1403-29072015} $\mbox{ }$\\
(i) $N\subset X$ is said to be \emph{$\Phi$-admissible} if and only if, for any $(t_n)$ in $[0,+\infty)$ with $t_n\to +\infty$
and $(v_n)$ in $X$ such that $\Phi_{[0,t_n]} (v_n) \subset N$, the sequence $\left( \Phi_{t_n}(v_n) \right)$ contains a convergent subsequence.\\
(ii) $N$ is said to be \emph{strongly $\Phi$-admissible} if $N$ is $\Phi$-admissible and $\Phi$ does not \emph{explode} in $N$, i.e.  $\Phi_{[0,T_{\bar u})}(\bar u) \subset N$ implies $T_{ \bar u } = + \infty$.
\end{definition}
We also need the notion of admissibility for the families of local semiflows.
\begin{definition}\label{1403-29072015'} $N\subset X$ is said to be \emph{$\{\Phi^k\}_{k\in K}$-admissible} if and only if, for any $(t_n)$ in $[0,+\infty)$ with $t_n\to +\infty$, $(k_n)$ in $K$
and $(v_n)$ in $X$ such that $\Phi_{[0,t_n]}^{k_n} (v_n) \subset N$, the sequence $\left( \Phi_{t_n}^{k_n}(v_n) \right)$ contains a convergent subsequence. If additionally $\Phi^k$ does not explode in $N$ for any $k\in K$, then $N$ is \emph{strongly $\Phi$-admissible}.
\end{definition}

Let ${\cal I}(X)$ be the family of pairs $(\Phi, K)$ where $\Phi$ is a local semiflow on a metric space $X$ and $K$ is an isolated invariant set having a strongly $\Phi$-admissible isolating neighborhood of $K$. The Conley homotopy index $h(\Phi, K)$ of $K$ relative to $\Phi$ for $(\Phi, K)\in {\cal I}(X)$ is defined by
 $h(\Phi, K):=[(B/B^-, [B^-])]$ where $B$ is an isolating block of $K$ (relative to $\Phi$) with the exit set $B^- \neq \emptyset$ and if $B^-=\emptyset$ we put
$h(\Phi, K)=[(B\cup \{ a\}, a)]$ where $a$ is an element that does not belong to $B$. In particular,
$h(\Phi, \emptyset)=\overline 0$ where $\overline 0:=[(\{a\},a)]$.\\

The Conley index has the following properties\\
(H1) For any $(\Phi, K)\in {\cal I}(X)$, if $h(\Phi, K)\neq \overline{0}$, then $K\neq \emptyset$.\\
(H2) If $(\Phi, K_1), (\Phi, K_2)\in {\cal I}(X)$ and $K_1\cap K_2=\emptyset$, then $(\Phi, K_1\cup K_2)\in {\cal I}(X)$
and $h(\Phi, K_1\cup K_2) = h(\Phi, K_1)\vee h(\Phi, K_2)$.\\
(H3) For any $(\Phi_1, K_1)\in {\cal I}(X_1)$ and $(\Phi_2, K_2)\in {\cal I}(X_2)$, $(\Phi_1\times \Phi_2, K_1\times K_2)\in {\cal I} (X_1\times X_2)$
and $h(\Phi_1\times \Phi_2, K_1\times K_2) = h(\Phi_1, K_1) \wedge h(\Phi_2, K_2)$.\\
(H4) If the family of semiflows $\{\Phi^{(\mu)} \}_{\mu\in [0,1]}$ is continuous and there exists $V$ such that
$V$ is strongly $\{ \Phi^{(\mu)} \}_{\mu\in [0,1]}$-admissible and $K_\mu = \mathrm{Inv}_{\Phi^{(\mu)}} (V) \subset \mathrm{int}\ V$, $\mu\in [0,1]$, then
$$
h(\Phi^{(0)}, K_0) = h(\Phi^{(1)}, K_1).
$$
In the linear case we shall use the following formula for computation of Conley index.
\begin{theorem} \label{11052015-1744}{\em (See \cite[Ch. I, Th. 11.1]{rybakowski})}
Suppose that $X$ is a normed space such that $X=X_- \oplus X_+$ with $k:=\dim X_+ <+\infty$ and a $C_0$-semigroup $\{T(t)\colon X\to X \}_{t\geq 0}$ is such that
$T(t)(X_+)\subset X_+$ and $T(t)(X_-)\subset X_-$,
$\|T(t)x\|\leq Ke^{-\alpha t}\|x\|$ for $x\in X_-$, $t\geq 0$ and some $\alpha>0$ and $\{T(t)\}_{t\geq 0}$ can be extended to a $C_0$-group $\{ T(t)\}_{t\in\R}$ on $X_+$ such that
$\|T(t)x\|\leq Ke^{\beta t} \|x\|$ for all $x\in X_+$, $t\leq 0$ and some $\beta>0$.
Then $\Phi\colon [0,+\infty) \times X\to X$ given by $\Phi(t,x):=T(t)x$ is a semiflow on $X$, the set $\{ 0\}$ is the set of all bounded full solutions of $\Phi$, $(\Phi, \{ 0 \})\in {\cal I}(X)$ and $h(\Phi, \{ 0\})=\Sigma^k$ where $\Sigma^k=[(S^k, \overline s)]$, i.e. it is the homotopy type of a $k$-dimensional sphere with a point.
\end{theorem}
We shall use the theory of irreducible sets due to Rybakowski \cite{rybakowski}.
Recall that an isolated invariant set $K$ (relative to a local semiflow $\Phi$) is called {\em reducible} if there exist isolated invariant sets
$K_1, K_2$ such that $K=K_1 \cup K_2$, $K_1\cap K_2=\emptyset$, $(\Phi, K_1), (\Phi, K_2)\in {\cal I}(X)$ and both
$h(\Phi, K_1)\neq \overline 0$ and $h(\Phi, K_2)\neq \overline 0$. We say that $K$ is {\em irreducible} if it is not reducible.
It is known that the set $K$ is irreducible if one of the following conditions is satisfied: $K$ is connected, $h(\Phi, K)=\overline{0}$ or
$h(\Phi, K)=\Sigma^k$ (see \cite[Ch. I, Th. 11.6]{rybakowski}). The  concept of irreducible set turns out to be useful due to the following
\begin{theorem}\label{15052015-0642} {\em (See \cite[Ch. I, Th. 11.5]{rybakowski})}
If $K_0\subset K\subset X$, $(\Phi,K), (\Phi, K_0)  \in {\cal I}(X)$, $K$ is irreducible and $\overline 0 \neq h(\Phi, K_0)\neq h(\Phi, K)$, then there exists a full solution
$\sigma\colon \R\to K$ such that $\sigma(\R)\not\subset K_0$ and either $\alpha(\sigma)\subset K_0$ or $\omega(\sigma)\subset K_0$.
\end{theorem}

\section{Properties of abstract evolution equations with $m$-accretive operators}

Let $A\colon D(A)\multimap X$ defined in a Banach space $X$ be an $m$-accretive operator, $f\colon [0,T]\to X$, $T>0$, and $\bar u\in \overline{D(A)}$.
We shall consider the equation
\begin{equation}\label{1208-11022015}
\begin{cases}
\dot u(t) \in - A u(t) + f(t),& t\in [0,T],\\
u(0)=\bar u.
\end{cases}
\end{equation}

\begin{definition}
A continuous function $u\colon [0, T]\to X$ is called an \em{integral solution} of {\em (\ref{1208-11022015})} in $X$ if and only if $u(0)=\bar u$ and
$$
\|u(t)-v\| \leq \|u(s)-v\| + \int_{s}^{t} [u(\tau)-v, f(\tau)-w]_s \d\tau
$$
for all $(v,w)\in \Gr(A)$ and $0\leq s<t\leq T$, where $[x,y]_s
=\inf_{h>0}(\|x+hy\|-\|x\|)/h$.
\end{definition}
\begin{remark}  \label{07042015-1652}
(i) It appears that for any $\bar u\in \overline{D(A)}$ and $f\in L^1 (0, T; X)$ the problem (\ref{1208-11022015}) admits a unique integral solution (see e.g. \cite{Barbu}). It may be also shown that this integral solution is also a mild solution (see e.g.  \cite{Barbu}), and as such, is a limit of discrete  approximations. \\
\indent (ii) Now consider Banach spaces $X, \tilde X$ such that $X$ is continuously embedded into  $\tilde X$. Suppose that a $m$-accretive operator $A$ in $X$ has an extension $\tilde A$ in $\tilde X$ such that it is $m$-accretive in $\tilde X$.
Then, for any $\bar u \in X$, $f\in L^1(0, T; X)$, the integral solution of { (\ref{1208-11022015})} is also an integral solution of { (\ref{1208-11022015})} in the space $\tilde X$. It follows immediately by use of discrete  approximations (see (i)).\\
\indent (iii) Let $\Sigma_A (\bar u,f)$ denote the integral solution of (\ref{1208-11022015}). Then, for any $\bar u_1,\bar u_2\in \overline{D(A)}$ and
$f_1, f_2\in L^1(0,T;X)$ one has (see \cite{Barbu})
$$
\|\Sigma_A (\bar u_1, f_1)(t)\!-\!\Sigma_A (\bar u_2,f_2)(t)\|\leq \|\bar u_1-\bar u_2\|+\int_{0}^{t}\! \|f_1(s)-f_2(s)\|\d s \mbox{ for } t\in [0,T].
$$
\\
\indent (iv) In particular, if we take $f=0$, then one may define a family of operators $\{S_A(t)\colon \overline{D(A)}\to \overline{D(A)}\}_{t\geq 0}$ by
$S_A(t)\bar u:=u(t)$ where $u$ is the integral solution of $\dot u(t) \in - Au(t)$, $t>0$, with $u(0)=\bar u$ (existing due to Remark \ref{07042015-1652} (i)).
It appears that $S_A (0) \bar u=\bar u$, for all $\bar u\in \overline{D(A)}$, and $S_A(t)\circ S_A(s)=S_A(t+s)$, for any $t,s\geq 0$, the mapping
$ [0,+\infty)\times \overline{D(A)} \ni (t,\bar u)\mapsto S_A(t)\bar u \in X$ is continuous and
$$
\|S_A(t)\bar u - S_A(t)\bar v\|\leq \|\bar u-\bar v\| \ \mbox{ for all } \bar u,\bar v\in \overline{D(A)}.
$$
\end{remark}

Let $A_n\colon D(A_n)\to X$, $n\geq 0$, be $m$-accretive operators. We say that $A_n$ {\em converges} to $A$ {\em in the sense of graphs} (and denote $A_n\stackrel{\Gr}{\to} A$) if and only if $\Gr A\subset \liminf\Gr A_n$. This is equivalent by \cite[Proposition 4.4]{Barbu-Springer} to the convergence
\[\lim_{n\to\infty}(I+\lambda A_n)^{-1}\bar u\to (I+\lambda A_0)^{-1}\bar u\text{ for all }\bar u\in X,\ \lambda>0.\]
We shall use the following continuity and compactness result.
\begin{proposition} \label{22042015-1020}
Let $A_n\colon D(A_n)\multimap X$, $n\geq 0$, be $m$-accretive operators.
\begin{enumerate}[\upshape (i)]
\item {\em (Trotter-Kato theorem)} If $\bar u_n\to \bar u_0$ in $X$, $u_n\in \overline{D(A_n)}$ for all $n\geq 0$, $A_n\stackrel{\Gr}{\to} A$ and $f_n\to f_0$ in $L^1(0,T;X)$, then $\Sigma_{A_n}(\bar u_n, f_n)\to \Sigma_{A_0}(\bar u_0, f_0)$ in $C([0,T],X)$.
\item Assume that $\overline{D(A_n)} =\overline{D(A_1)}$ for any $n\geq 1$ and
the set $\bigcup_{n\geq 1} S_{A_n} (t) (B)$ is relatively compact for any bounded $B\subset \overline{D(A_1)}$ and $t>0$.
 If $\{\bar u_n\}_{n\geq 1} \subset \overline{D(A_1)}$ is bounded
 and $\{f_n \}\subset L^1(0,T; X)$ is uniformly integrable, then for any $t\in(0,T]$ the set $\{\Sigma_{A_n}(\bar u_n, f_n) (t)  \}_{n\geq 1}$ is relatively compact.
\end{enumerate}
\end{proposition}
\begin{proof} (i) can be found in \cite[Ch. 4, Th. 4.14]{Barbu-Springer} and (ii) is proved in \cite[Prop. 2.26]{Cwisz-JEE}. \end{proof}

We shall also consider nonlinear problems of the form
\begin{equation}\label{07042015-1656}
\left\{
\begin{array}{l}
\dot u(t) \in -  A u(t) + F(u(t)),\ t\in [0,T],\\
u(0)=\bar u
\end{array}
\right.
\end{equation}
with a locally Lipschitz $F\colon X\to X$ and $\bar u\in \overline{D(A)}$. We shall say that a continuous $u\colon [0,T]\to X$ ($T>0$) is an integral solution of (\ref{07042015-1656}) if $u$
is an integral solution of (\ref{1208-11022015}) with $f\colon [0,T]\to X$ given by $f(t):=F(u(t))$ for $t\in [0, T]$, that is, if $u=\Sigma_{A}(\bar u,F\circ u)$. Let us state a general existence and uniqueness theorem.
\begin{proposition} \label{0208-29042015} $\mbox{ }$\\
{\em (i)} \parbox[t]{144mm}{For any  $\bar u\in \overline{D(A)}$ there exists a continuous $u\colon [0, T_{\bar u} )\to X$ with $T_{\bar u} \in (0, +\infty]$, called a maximal integral solution, such that for each $T\in (0,T_{\bar u})$ the function $u|_{[0, T]}$ is an integral solution of {\em (\ref{07042015-1656})}. Moreover, either $T_{\bar u} = +\infty$ or $T_{\bar u}<+\infty$ and $\limsup_{t\to T_{\bar u}^-} \|u(t)\| = +\infty$.}\\[0.5em]
{\em (ii)} \parbox[t]{144mm}{Let $A_n\colon D(A_n)\multimap X$, $n\geq 0$, be $m$-accretive operators with $\overline{D(A_n)}=\overline{D(A_1)}$ for all $n\geq 1$ and $A_n\stackrel{\Gr}{\to} A$. Let $F_n\colon X\to X$ have common Lipschitz constants on bounded subsets of $X$. Assume that $F_n (\bar u) \to F_0 (\bar u)$ for any $\bar u\in X$.
Let $u_n\colon [0,T_{\bar u_n})\to X$ for each $n\geq 0$ be the maximal integral solution of
\begin{equation}\label{22072015-0808}
\begin{cases}
\dot u(t)\in - A_n u(t)+F_n(u(t)),& t>0,\\
u(0)=\bar u_n
\end{cases}
\end{equation}
where $u_n\in \overline{D(A_n)}$ for all $n\geq 1$ and $\bar u_n \to \bar u_0$ as $n\to +\infty$. Then $\liminf_{n\to +\infty} {T_{\bar u_n}} \geq T_{\bar u_0}$ and $u_n\to u_0$ in $C([0,T], X)$ for any $T \in (0, T_{\bar u_0})$.}\\[0.5em]
{\em (iii)} \parbox[t]{144mm}{Assume that $m$-accretive operators $A_n\colon D(A_n)\multimap X$  with $\overline{D(A_n)}=\overline{D(A_1)}$, $n\geq 1$,  and locally Lipschitz $F_n\colon X\to X$, $n\geq 1$,  are  such that for any bounded $B\subset X$
\[
\bigcup_{n\geq 1} S_{A_n}(t) (B\cap \overline{D(A_n)}) \mbox{ is relatively compact and } \bigcup_{n\geq 1} F_n(B) \mbox{ is bounded.}
\]
If $u_n\colon  [0,T]\to X$, $n\geq 1$, are integral solutions of {\em (\ref{22072015-0808})} and there exists $R>0$ such that $\|u_n(t)\|\leq R$ for all $t\in [0,T]$ and $n\geq 1$, then the set $\{ u_n (t) \}_{n\geq 1}$ is relatively compact for any $t\in (0,T]$.}
\end{proposition}

\begin{proof} (i) is classic (see \cite{Barbu}).

(ii) Fix $T\in (0,T_{\bar u_0})$. Let $R:=\max_{t\in [0,T]} \|u_0(t)\|$ and let $L_R$ be the common Lipschitz constant for $F_n$, $n\geq 0$, on the ball $B(0,3R)$. Define
\[\alpha_n:=\sup_{0\leq t\leq T}\|\Sigma_{A_n}(\bar u_n, F_n \circ u_0)(t)-\Sigma_{A_0}(\bar u_0, F_0 \circ u_0)(t)\|.\] From Proposition \ref{22042015-1020} (i) it follows that $\alpha_n\to 0^+$ and therefore $\alpha_n e^{T L_R } < R$ for $n\geq n_0$ for some $n_0\in\N$. Fix $n\geq n_0$. In order to prove the inequality from the conclusion it suffices to show that $T< T_{\bar u_n}$. Suppose, contrary to our claim, that $T_{\bar u_n}\leq T$. Therefore $u_n([0,t])\subset B(0,3R)$ and $\|u_n(t)\|>2R$ for some $t<T$, which follows from (i).

Using Remark \ref{07042015-1652} (iii) one has, for $\tau\in [0,t]$,
\begin{align*}
\|u_n(\tau)-u_0(\tau)\|&\leq
\|\Sigma_{A_n}(\bar u_n, F_n \circ u_n)(\tau)-\Sigma_{A_n}\!(\bar u_n,  F_n \circ u_0)(\tau)\| \\
& +\|\Sigma_{ A_n} (\bar u_n, F_n \circ u_0)(\tau)-\Sigma_{A_0} (\bar u_0, F_0\circ  u_0)(\tau)\|\\
&\leq \alpha_n+L_R \int_{0}^{\tau}\|u_n(s)-u_0(s)\|\d s.
\end{align*}
In consequence, by the Gronwall inequality we obtain the estimate
\begin{equation} \label{t6}
\|u_n (\tau)-u_0(\tau)\|\leq \alpha_n e^{T L_R } < R \ \mbox{ for } \tau\in [0,t].
\end{equation}
This means that $\|u_n(t)\|\leq \|u_0(t)\|+R\leq 2R$. This contradicts our assumption.

By the estimate $\|u_n (t)-u_0(t)\|\leq \alpha_n e^{T L_R }$ for $t\leq T$, the convergence from the second part of the conclusion holds true.

(iii) is a direct consequence of Proposition \ref{22042015-1020} (ii) with $f_n:= F_n\circ u_n$.

\end{proof}

\begin{remark}\label{24022016-1242}
Let $D$ be the set of $(t,\bar u)\in [0,+\infty)\times X$ such that the problem (\ref{07042015-1656}) has a solution on $[0,t]$ and let $\Phi\colon D\to X$ be defined by $\Phi(t,\bar u)=\Phi_t (\bar u):=u(t)$, where $u\colon [0,t]\to X$ is the integral solution of (\ref{07042015-1656}). It clearly follows from Proposition \ref{0208-29042015} that $\Phi$ is a local semiflow on $X$.
\end{remark}

Now suppose that $H$ is a Hilbert space with the scalar product $\langle \cdot, \cdot \rangle$ and the norm $\|\cdot\|$ and consider a lower semicontinuous convex functional $\varphi\colon  X\to [0,+\infty]$ with $D(\varphi):=\{u\in H\mid \varphi(u)<+\infty \}$ and $\varphi(0)=0$. Recall that the {\em subdifferential} $\partial \varphi\colon D(\partial \varphi) \to H$ of $\varphi$ is defined by $D(\partial \varphi):= \{ u\in D(\varphi) \mid \mbox{ there exists } \xi\in H \mbox{ such that } \langle \xi, v- u \rangle \leq \varphi (v) - \varphi (u) \mbox{ for all } v\in H\}$ and $\partial \varphi(u):=\{\xi \in H \mid \langle \xi, v- u \rangle \leq \varphi(v)-\varphi(u) \mbox{ for all } v\in H\}$. It is known that $\partial \varphi$ is a $m$-accretive operator in $H$. Hence, if $f\in L^2(0,T;H)$ and $\bar u\in H$, one can consider the following problem
\begin{equation} \label{07042015-1749}
\left\{
\begin{array}{l}
\dot u(t) + \partial \varphi u(t) \ni f (t),\ t>0,\\
u(0)=\bar u.
\end{array}
\right.
\end{equation}
It appears that integral solutions in this case are more regular and are strong solutions.
\begin{proposition} \label{08042015-1552} {\em (\cite[Ch. 4, Th. 4.11 and Lem. 4.4]{Barbu-Springer})}
If $f\in L^2(0, T; H)$, $\bar u\in H$ and $u\in C([0,T],H)$ is an integral solution of {\em (\ref{07042015-1749})}, then
$u$ is a.e. differentiable on $(0,T)$ and has the following properties\\
{\em (i)} $u(t)\in D(\partial \varphi)$ and $\dot u(t) + \partial \varphi (u(t))\ni f(t)$ for a.e. $t\in (0, T)$;\\
{\em (ii)} $\left(t\mapsto  t^{1/2} \dot u(t)\right) \in L^2 (0, T; H)$ and $\varphi\circ u \in L^1(0, T)$;\\
{\em (iii)} \parbox[t]{144mm}{if moreover $\bar u \in D(\varphi)$, then $u(t)\in D (\varphi)$, for all $t\in [0,T]$, the function
$\varphi \circ u\colon [0,T]\to \R$ is absolutely continuous,
$$
\frac{1}{2}\int_{0}^{t} \|\dot u(s)\|^2 \d s + \varphi(u(t)) \leq \frac{1}{2}\int_{0}^{t} \|f(s)\|^2 \d s +\varphi (\bar u), \ \mbox{ for all } t\in [0,T],
$$
and
$$
\frac{d}{dt}\varphi (u(t))=-\|\dot u(t)\|^2 + (\dot u(t), f(t)) \ \mbox{ for a.e. } t\in [0,T].
$$
}\\
In particular, for all $\bar u\in H$ the function $\varphi\circ u$ is continuous on $(0,T]$.
\end{proposition}
In order to estimate time derivative of solutions we shall need the following result.
\begin{proposition} {\em (See \cite[Th. 4.12]{Barbu-Springer})}\label{12042015-1740}
Assume that $\bar u \in \overline{D(\partial \varphi)}$  and $f\in W^{1,1}([0,T], H)$.
If $u$ is the integral solution of {\em (\ref{07042015-1749})} with $u(0)=\bar u$, then, for all $t\in (0,T]$, $u(t)\in D(\partial \varphi)$,
$$
\dot u_{+}(t) + (\partial \varphi)(u(t))\ni f(t)
$$
and
\begin{equation}\label{15042015-1334}
t^2 \|\dot u (t)\|^2 \leq \int_{0}^{t} s\|f(s)\|^2 \d s + 2\left(\|\bar u\| + \int_{0}^{t} \|f(s)\| \d s \right)^2 + \frac{t^2}{2} \left( \int_{0}^{t} \| \dot f(s) \| \d s\right)^2.
\end{equation}
In particular, if $f\equiv 0$, then $t\|\dot u(t)\|\leq\sqrt 2\|\bar u\|$.
\end{proposition}
\noindent We shall need the following compactness criterion being an extension of \cite[Ch. 4, Th. 2.4]{Barbu} to a family of semigroups generated by subdifferentials.
\begin{proposition} \label{23042015-2220}
Let $\varphi_n\colon  H\to [0,+\infty]$, $n\geq 1$, are lower semicontinuous convex functions such that $D(\varphi_n):=\{u\in H\mid \varphi_n(u)<+\infty \}$ is dense in $H$ and $\varphi_n (0)=0$, $n\geq 1$. If for any $\lambda>0$  the set $L_\lambda:=\bigcup_{n\geq 1} \{ u\in H \mid \varphi_n (u) \leq  \lambda \}$ is relatively compact, then the family of semigroups $\{ S_{\partial \varphi_n} (t)\colon H\to H\}_{t\geq 0}$ is compact, i.e. the set $\bigcup_{n\geq 1} S_{\partial \varphi_n} (t) (B)$ is relatively compact for all $t>0$ and any bounded $B\subset H$.
\end{proposition}
\noindent In order to prove Proposition \ref{23042015-2220} we need the following lemma.
\begin{lemma}\label{lem:fi_bdd} Under the assumptions of Proposition {\em\ref{23042015-2220}}, for all $n\geq 1$, $t>0$ and $\bar u\in H$ we have
\[
\varphi_n(S_{\partial \varphi_n}(t)\bar u)\leq \max\left\{1,\frac {M^2}{t}\right\} \text{ with } M:=\|\bar u\|+\sup_{\bar v\in L_0} \|\bar v\|.
\]
\end{lemma}
 \begin{proof} Fix $n\geq 1$, $\bar u\in H$. Put $u:=S_{\partial \varphi_n}(\cdot)(\bar u)$. By use of Proposition \ref{08042015-1552} (iii) we obtain
\begin{equation}\label{t1}
\frac{\d}{\d t}\left(\varphi_n (u(t))\right)=-\left\|\dot u (t)\right\|^2\leq 0 \ \mbox{ for a.e. } t>0,
\end{equation}
which implies in particular that $\varphi_n (u(t)) \leq \varphi_n (\bar u)$ for all $t\geq 0$. From the fact that,
 $-\dot u(t)\in \partial\varphi_n(u(t))$ and $\varphi_n(0)=0$ it follows that $\varphi_n(u(t)) \leq \| \dot u (t) \| \|u(t)\|$.
Using the contractivity of $\{ S_{\partial \varphi_n} (t) \}_{t\geq 0}$ one has, for a.e. $t>0$,
\begin{equation}\label{t2}
\varphi_n(u(t)) \leq \| \dot u (t) \| \|u(t)\| \leq \|\dot u (t) \|  \left( \| S_{\partial \varphi_n}(t)0 \| + \|\bar u\| \right)
\leq M \| \dot u (t) \|.
\end{equation}
Combining \eqref{t1} and \eqref{t2} we obtain
\begin{equation}\label{27072015-1156}
\frac{\d}{\d t}\left(\varphi_n (u(t))\right) \leq -\frac{1}{M^2} \left( \varphi_n (u(t))\right)^2 \mbox{ for a.e. } t>0.
\end{equation}
If $\varphi_n(\bar u)\leq 1$, then $\varphi_n (u(t)) \leq \varphi_n (\bar u) \leq 1$ for all $t>0$. If $\varphi_n (\bar u) >1$, then
$$
\frac{\d}{\d t} \left( \frac{1}{\varphi_n(u(t))} \right)\geq \frac{1}{M^2}
$$
for a.e. $t>0$ as long as $\varphi(u_n(t))>0$. This implies
$$
\frac{1}{\varphi_n(u(t))}\geq \frac{1}{\varphi_n(\bar u)} + \frac{t}{M^2} > \frac{t}{M^2},
$$
which ends the proof.
\end{proof}
\begin{proof}[Proof of Proposition \ref{23042015-2220}.] Fix $t>0$ and let $B\subset L^2(0,l)$ be bounded. From Lemma \ref{lem:fi_bdd} we obtain $S_{\partial \varphi_n}(t)(B)\subset L_\lambda$ for some $\lambda$ independent of $n$. The conclusion follows from the relative compactness of $L_\lambda$.
\end{proof}

\section{Existence and regularity for $p$-Laplace evolution equation}

First we put the equation (\ref{07042015-1828}) in the abstract framework to precise the concept of solution and get information on their regularity. To this end, define
$$
D({\mathbf A}_p):=\{ u\in C_0 (0,l)\cap W_{loc}^{1,p-1} (0,l) \mid (|u'|^{p-2}u')' \in C_0(0,l)\},
$$
$$
{\mathbf A}_p u:= -(|u'|^{p-2}u')',\ \  u\in D({\mathbf A}_p).
$$
Observe that if $u\in D(\mathbf{A}_p)$, then $u\in C^1(0,l)$ and $|u'|^{p-2}u'\in C^1(0,l)$.
\begin{lemma}
The operator ${\mathbf A}_p$  is $m$-accretive.
\end{lemma}
\begin{proof} It can be demonstrated along the lines of \cite[Lemma 6.1]{cw-kr}.  \end{proof}
\begin{definition}\label{15042015-0945}
Fix $\bar u\in C_0(0,l)$ and $T>0$. By a \emph{solution} of \eqref{07042015-1828} on the interval $[0, T]$, with the initial condition $u(x,0)=\bar u(x)$ for $x\in [0,l]$, we mean the integral solution of
\begin{equation}\label{08042015-0337}
\dot u(t) = - {\mathbf A}_p u(t) + {\mathbf F} (u(t)),\  t\in [0,T],
\end{equation}
with $u(0)=\bar u$, where ${\mathbf F}\colon C_0(0,l)\to C_0 (0,l)$ is given by ${\mathbf F}(u)(x):=f(x,u(x))$, $x\in [0,l]$, $u\in C_0 (0,l)$.
\end{definition}
\begin{theorem}\label{29072015-1333}
Assume that a continuous $f\colon [0,l]\times \R\to\R$ satisfies the local Lipschitz condition {\em (\ref{0128-29042015})} and that $f(x,0)=0$ for all $x\in [0,l]$. Then, for any $\bar u\in C_0 (0,l)$ there exists a unique function $u\in C([0,T_{\bar u}^{(p,f)}), C_0 (0,l))$ with $T_{\bar u}^{(p,f)} \in (0,+\infty]$ such that $u(0)=\bar u$, the function $u|_{[0,T]}$  is the solution of {\em (\ref{07042015-1828})}, for all $T\in (0, T_{\bar u}^{(p,f)}]$, and either $T_{\bar u}^{(p,f)}=+\infty$ or $T_{\bar u}^{(p,f)}<+\infty$ and
$\limsup_{t\to T_{\bar u}^{(p,f)-}} \|u(t)\|_{\infty}=+\infty$.
\end{theorem}
\begin{proof} The proof follows directly from Proposition \ref{0208-29042015} applied to (\ref{08042015-0337}). \end{proof}

So as to study regularity properties of solutions, it will be convenient to use also an $L^2$-extension of ${\mathbf A}_p$. We shall consider $\bar{\mathbf A}_p\colon D(\bar{\mathbf A}_p) \to L^2(0,l)$ given by
$$D(\bar{\mathbf A}_p):=\{u\in W_{0}^{1,p} (0,l) \mid (|u'|^{p-2}u')' \in L^2(0,l)\} \text{ and }
\bar{\mathbf A}_p u:= -(|u'|^{p-2}u')'.
$$
 Clearly, $D(\bar{\mathbf A}_p)\subset C^1(0,l)$ and $\Gr({\mathbf A}_p) \subset \Gr (\bar{\mathbf A}_p)$ (by the embedding of $C_0(0,l)$ into $L^2(0,l)$).
\begin{lemma} \label{08042015-1551} $\mbox{ }$\\
{\em (i)} \parbox[t]{144mm}{The operator $\bar{\mathbf A}_p$ is $m$-accretive and there exists $c_p>0$ such that
\begin{equation}\label{15042015-1337}
(\bar {\mathbf A}_p u- \bar {\mathbf A}_p v, u-v)_{L^2} \geq c_p \|u-v\|_{W_0^{1,p}}^{p} \mbox{ for all } u,v\in D(\bar {\mathbf A}_p).
\end{equation}}\\[1em]
{\em (ii)} \parbox[t]{144mm}{$\bar{\mathbf A}_p= \partial\varphi_p$,   where $\partial\varphi_p$ is the subdifferential of the lower semicontinuous convex functional $\varphi_p\colon  L^2(0,l)\to \R\cup \{ + \infty \}$
given by
$$
\varphi_p (u)= \left\{
\begin{array}{ll}
\frac{1}{p} \int_{0}^{l} |u'(x)|^p \d x &   \mbox{ if } u\in W_{0}^{1,p}(0,l)\\
 +\infty &  \mbox{ if } u\in L^2(0,l)\setminus W_{0}^{1,p}(0,l).
 \end{array}\right.
 $$}\\[1em]
{\em (iii)} \parbox[t]{144mm}{For any $M,T>0$ there exists $C=C(M,T)>0$  (independent of $p\geq 2$)  such that $\|S_{\bar {\mathbf A}_p}(t)\bar u_1- S_{\bar{\mathbf A}_p}(t)\bar u_0\|_{W^{1,p}_0}\leq C(T,M)\cdot  \|\bar u_1-\bar u_0\|_{L^2}^{1/p}$ whenever $\|\bar u_0\|_{L^2},\|\bar u_1\|_{L^2}\leq M$ and $t\geq T$.}
\end{lemma}
 \begin{proof} (i) If we change the space $L^p(0,l)$ with $L^2(0,l)$ in the proof of \cite[Prop. 4.1]{cw-mac}, we obtain that the operator $\bar {\mathbf A}_p$ is maximal monotone, which in Hilbert spaces is equivalent to being $m$-accretive. In particular, we have  the estimate (\ref{15042015-1337}).\\
\indent (ii)  Take $(u,v)\in \mathrm{Gr}\, \bar{\mathbf A}_p$, which means that $u\in W^{1,p}_0(0,l)$ and $v=-(j'(u'))'$, where $j(s)=\frac 1p|s|^p$. Therefore, for all $w\in W^{1,p}_0(0,l)$ we have
\[(v,w-u)_{L^2}=\int_0^lv(w-u)=\int_0^l j'(u')(w'-u')\leq \int_0^lj(w')-\int_0^lj(u')=\varphi_p(w)-\varphi_p(u),\] which shows that $(u,v)\in\mathrm{Gr}(\partial\varphi_p)$. In other words $\Gr (\bar {\mathbf A}_p)\subset \Gr(\partial \varphi_p)$. From the general theory of maximal operators, it follows that, since $\bar {\mathbf A}_p$ is maximal monotone in a Hilbert space, its graph is maximal among graphs of accretive operators. Therefore we get $\Gr(\partial \varphi_p)=\Gr (\bar {\mathbf A}_p)$.\\
\indent (iii) Let $v_i:= S_{\bar{\mathbf A}_p}(\cdot)\bar u_i$ for $i=0,1$ and let $w=v_1-v_0$. From Proposition \ref{12042015-1740} it can be concluded that $w(s)\in W^{1,p}_0(0,l)$ and that $\|\dot w(s)\|_{L^2}\leq \sqrt 2 s^{-1}\cdot(\|\bar u_0\|_{L^2}+\|\bar u_1\|_{L^2})$ for all $s>0$. From (i) it follows that
\begin{align*} c_p\|w(t)\|_{W^{1,p}_0}^p\leq \|\dot w(t)\|_{L^2}\cdot \|w(t)\|_{L^2}\leq \frac{2\sqrt 2M}{t}
\|\bar u_1-\bar u_0\|_{L^2},
\end{align*}
which gives the conclusion  as $c_p=2^{2-p}\geq 2^{-p}$.
\end{proof}
\noindent The following result sheds more light on the regularity of solutions.
\begin{theorem}\label{06052015-1149}
If $u\in C([0,T], C_0(0,l))$ is a solution of  {\em (\ref{07042015-1828})}, then
\[u\in C((0,T], W_{0}^{1,p}(0,l)) \cap W^{1,2}((0,T];L^2(0,l)),\]
\[ u(t)\in D(\bar{\mathbf A}_p)\text{ and }
\dot u(t) = - \bar {\mathbf A}_p u(t) + \bar {\mathbf F} (u (t)) , \mbox{ for a.e. } t\in [0,T],
\]
where $\bar {\mathbf F}\colon  C_0( 0,l) \to L^2 (0,l)$ is given by $\bar {\mathbf F}(u)(x):= f(x,u(x))$ for a.e. $x\in (0,l)$ and
\[W^{1,2}((0,T];L^2(0,l))=\left\{u\in L^2\left(0,T;L^2(0,l)\right)\mid\dot u\in L^2_{loc}\left((0,T];L^2(0,l)\right) \right\}.\]
Moreover, the functional $\varphi_{p,f}\colon W^{1,p}_0(0,l)\to\mathbb{R}$ given by
\[ \varphi_{p,f}(u):=\frac 1p\int_{0}^{l} |u'(x)|^p \d x - \int_{0}^{l} {\cal F} (x,u(x)) \d x, \  u\in W^{1,p}_0(0,l),
\]
where ${\cal F} (x,u):= \int_{0}^{u} f(x,\tau) \d \tau$,  is a Lyapunov function for {\em (\ref{07042015-1828})}, since for any solution $u\in C([0,T], C_0(0,l))$ and $0<s<t<T$ one has
\begin{equation}\label{15052015-1726}
{\mathbf \varphi}_{p,f}  (u(t))-{\mathbf \varphi}_{p,f}(u(s)) = - \int_{s}^{t} \int_{0}^{l} |\dot u (\tau)|^2 \d x \d\tau.
\end{equation}
\end{theorem}
\begin{proof} By Remark \ref{07042015-1652}, the function $u$ may be viewed as an element of $C([0,T], L^2(0,l))$ and the integral solution (in $L^2(0,l)$) of
$$
\dot u(t) = - \bar {\mathbf A}_p u(t) + {\mathbf f}(t), \, t\in [0,T],
$$
with ${\mathbf f}\in L^2(0,T; L^2(0,l))$ given by ${\mathbf f}(t):={\mathbf F}(u(t))$, $t\in [0,T]$.
This together with Lemma \ref{08042015-1551} (ii) and Proposition \ref{08042015-1552},
gives that $u(t)\in D(\partial {\mathbf \varphi}_p ) \subset D({\mathbf\varphi}_p)= W_{0}^{1,p}(0,l)$ and
$$
\dot u(t) = -\partial {\mathbf \varphi_p} (u(t))+{\mathbf f}(u(t)),
$$
for a.e. $t\in [0,T]$, and that $\varphi_p\circ u$ is continuous on $(0,T]$. Therefore, the function  $(0,T] \ni t \mapsto \|u(t)\|_{W^{1,p}}$ is continuous. Since $u\in C([0,T],L^2(0,l))$, using the uniform convexity of $W_{0}^{1,p}(0,l)$ we obtain that $u\colon (0,T]\to W_{0}^{1,p}(0,l)$ is continuous.\\
\indent In order to verify that $\varphi_{p,f}$ is a Lyapunov function observe that, by Proposition \ref{08042015-1552}, one has  $u\in W^{1,2}((0,T];L^2(0,l))$ and
\begin{equation}\label{29072015-1354}
\frac{\d}{\d t}\left( \varphi_p(u(t))\right) - (\bar {\mathbf F}(u(t)), \dot u(t))_{L^2} = -\|\dot u(t)\|_{L^2}^{2}\ \text { for a.e. } t\in [0,T].
\end{equation}
Now take any $t\in [0,T]$ such that $\dot u(t)$ exists (in $L^2(0,l)$) and any sequence $(h_n)$ in $\R\setminus \{ 0 \}$ with $h_n\to 0$.
Passing to a subsequence, if necessary, we may suppose that $(u(t+h_n)-u(t))/h_n \to \dot u(t)$ a.e. on $[0,l]$ and that
there is $g\in L^1(0,l)$ such that $|(u(t+h_n)-u(t))/h_n|\leq g$ a.e. on $[0,l]$. By means of  Lebesgue's dominated convergence theorem we have
\[({\cal F}(x,u(t+h_n)(x))-{\cal F}(x,u(t)(x)))/h_n \to f(x,u(t)(x)) \dot u(t)(x)\text{  a.e. on }[0,l].\]
Furthermore, since ${\cal F}$ is Lipschitz with respect to the second variable on bounded sets, for  $f$ is bounded on bounded sets, we can use dominated convergence theorem to get
$$
\frac{\d}{\d t}\left( \int_{0}^{l} {\cal F}(x,u(t)(x)) \d x \right) =   \int_{0}^{l} f(x,u(t)(x))  \dot u(t)(x)\d x = (\bar {\mathbf F}(u(t)), \dot u(t))_{L^2}.
$$
This together with (\ref{29072015-1354}) ends the proof.
\end{proof}
We shall also need some compactness and continuity of solutions with respect to the initial data.
\begin{theorem}\label{1502-29072015}
Suppose that $u_n\colon [0,T] \to C_0(0,l)$ with $u_n(0)=\bar u_n$, $n\geq 0$, are solutions of {\em (\ref{07042015-1828})} and there is $R>0$ such that $\|u_n(t)\|_\infty \leq R$ for all $t\in [0,T]$ and $n\geq 0$.\\
{\em (i)} \parbox[t]{144mm}{{\em (Continuity)} If $\bar u_n \to \bar u_0$ in $C_0 (0,l)$, then $u_n\to u_0$ in $C([0,T], C_0(0,l))$. If, additionally,  $(\bar u_n)$ is bounded in $W_0^{1,p}(0,l)$, then $u_n(t) \to u_0(t)$ in $W_{0}^{1,p}(0,l)$ for any $t\in (0,T]$.}\\[1em]
{\em (ii)} \parbox[t]{144mm}{{\em (Compactness)} The set $\{ u_n(t)\}_{n\geq 1}$ is relatively compact in $C_0(0,l)$ for any $t\in (0,T]$.}
\end{theorem}
\noindent Before proving the general theorem we shall prove it for the contraction semigroup generated by ${\mathbf A}_p$.
\begin{lemma}
For any $t>0$ and bounded $B\subset C_0 (0,l)$, the set $S_{{\mathbf A}_p} (t) (B)$ is a relatively compact subset of $W_{0}^{1,p}(0,l)$.
In particular, the semigroup $\{ S_{{\mathbf A}_p} (t)\colon C_0 (0,l)\to C_0 (0,l)\}_{t\geq 0}$ is compact.\\
\end{lemma}
\begin{proof}
First observe that for any $\lma\in \R$, $L_\lma:=\{ \bar  u\in L^2(0,l)\mid \varphi_p (\bar u)\leq \lma\}$ is bounded in $W_{0}^{1,p}(0,l)$, which by the Rellich-Kondrachov embedding theorem means that $L_\lma$ is relatively compact in $L^2(0,l)$.
Therefore, due to Proposition \ref{23042015-2220}, the semigroup $\{S_{\bar{\mathbf A}_p}(t)\colon L^2(0,l)\to L^2(0,l) \}_{t\geq 0}$ is compact.\\
\indent Now take a bounded set $B\subset C_0(0,l)$, $t>0$ and any sequence $(\bar u_n)$ in $B$.
Put $\alpha=t/3$. Since the set $\{ S_{{\mathbf A}_p}(\alpha)\bar u_n\}_{n\geq 1}$ is relatively compact in $L^2(0,l)$, without any loss of generality we may assume that  $S_{{\mathbf A}_p}(\alpha)\bar u_n \to \tilde v_0$ in $L^2 (0,l)$ for some $\tilde v_0 \in L^2 (0,l)$.
Put $\bar v_n := S_{{\mathbf A}_p} (2\alpha) \bar u_n$ as well as $\bar v_0:=S_{\bar {\mathbf A}_{p}} (\alpha)\tilde v_0$.
Clearly,  $\bar v_n\to \bar v_0$ in $L^2(0,l)$ and, by Lemma \ref{08042015-1551} (ii) and Proposition \ref{12042015-1740}, we see that $\bar v_0\in D(\part \varphi_p)\subset D(\varphi_p)= W^{1,p}_0(0,l)$, which implies $\bar v_0\in C_0 (0,l)$.\\
\indent In view of Lemma \ref{08042015-1551} (iii), one has
\begin{align*}
\|S_{{\mathbf A}_p}(t) \bar u_n - S_{{\mathbf A}_p}(\alpha) \bar v_0 \|_{W_{0}^{1,p}}^p
=  \|S_{{\mathbf A}_{p}} (\alpha)\bar v_n - S_{{\mathbf A}_{p}}(\alpha) \bar v_0\|_{W_{0}^{1,p}}^p
\leq C\cdot \|\bar v_n-\bar v_0\|_{L^2}
\end{align*}
for some constant $C>$. In consequence,  $S_{\mathbf{A}_p}(t)\bar u_n  \to  S_{\mathbf{A}_p} (\alpha)\bar v_0$ in $W_{0}^{1,p} (0,l)$, hence in $C_0 (0,l)$.
\end{proof}
\begin{proof}[Proof of Theorem \ref{1502-29072015}]
(i)  First observe that, due to Proposition \ref{0208-29042015} (ii), $u_n\to u_0$ in $C([0,T], C_0(0,l))$.
To prove the other part of the assertion assume that $(\bar u_n)$ is bounded in $W_{0}^{1,p}(0,l)$. Put $\tilde f(x,u)=f(x,r(u))$, where $r\colon\R\to[-R,R]$ is a metric projection. Let $\bar{\mathbf F}\colon C_0(0,l)\to L^2(0,l)$ and $\tilde{\mathbf F}\colon L^2(0,l)\to L^2(0,l)$ be the Nemytskii operators generated by $f$ (as in Theorem \ref{06052015-1149}) and $\tilde f$, respectively. Define a bounded sequence of elements ${\mathbf f}_n\in L^2(0,T;L^2(0,l))$, $n\geq 1$, by ${\mathbf f}_n:= \bar{\mathbf F}\circ u_n=\tilde{\mathbf F}\circ u_n$. Now observe that in view of Proposition \ref{08042015-1552} (iii) and the boundedness of
$(\bar u_n)$ in $W_0^{1,p}(0,l)$ and $({\mathbf f}_n)$ in $L^2(0,T;L^2(0,l))$ we have $u_n\in W^{1,2}((0,T);L^2 (0,l))$ and there exists $\tilde R>0$ such that
$\|\dot u_n\|_{L^2(0,T;L^2(0,l))}\leq \tilde R$ for all $n\geq 1$.\\
\indent Since $\tilde f$ is Lipschitz with respect to the second variable uniformly with respect to $x$, $\tilde {\mathbf F}$ is Lipschitz. Denote its Lipschitz constant by $\tilde L$. In consequence, for all $n\geq 1$ and $t,s\in [0,T]$, we get
\begin{equation}\label{0216-20150807}
\| {\mathbf f}_n (t) - {\mathbf f}_n (s) \|_{L^2} \leq {\tilde L} \| u_n(t) - u_n(s) \|_{L^2} \leq  {\tilde L} \int_{s}^{t} \|\dot u_n(\tau) \|_{L^2} \d\tau.
\end{equation}
By \cite[Thm 1.4.40]{Cazenave-Haraux}, this implies that
$\|\dot {\mathbf f}_n\|_{L^2(0,T;L^2(0,l))}\leq {\tilde L}\|\dot u_n\|_{L^2(0,T;L^2(0,l))}\leq {\tilde L} \tilde R$. Now fix $t\in (0,T]$.
By Proposition \ref{12042015-1740}, the sequence $(\dot u_n (t))$ is bounded in $L^2(0,l)$. Applying (\ref{15042015-1337}) we get, for all $n,m\geq 1$,
$$
\|u_n(t)-u_m(t)\|_{W_{0}^{1,p}}^{p} \leq c_{p}^{-1} \| \partial \varphi_p (u_n(t)) -  \partial \varphi_p (u_m(t))\|_{L^2} \|u_n (t)-u_m (t)\|_{L^2}.
$$
Since the values
\[ \|\partial \varphi_p (u_n(t)) -  \partial \varphi_p (u_m(t))\|_{L^2} \leq \|{\mathbf f}_n(t)\|_{L^2} +\|{\mathbf f}_m(t)\|_{L^2} + \| \dot u_n(t)\|_{L^2} +\| \dot u_m (t)\|_{L^2}
\] are bounded, we see that $(u_n(t))$ is a Cauchy sequence in $W_{0}^{1,p}(0,l)$, which means that $u_n(t)\to u_0(t)$ in $W_0^{1,p}(0,l)$.\\
\indent (ii) follows immediately from Proposition \ref{0208-29042015} (iii).
\end{proof}

We shall summarize the obtained results in the context of dynamical systems. Define ${\mathbf \Phi}^{(p,f)}\colon  {\mathbf D}^{(p,f)}\to {\mathbf X}$ where ${\mathbf X}:=C_0(0,l)$, by
\[
{\mathbf D}^{(p,f)}:= \left\{ (t,\bar u)\in {[0,+\infty)\times\mathbf X} \mid t< T_{\bar u}^{(p,f)} \right\} \text{ and }
{\mathbf \Phi}^{(p,f)} (t,\bar u) = {\mathbf \Phi}_t^{(p,f)} (\bar u):=u(t),
\]
where $u\colon  [0, T_{\bar u}^{(p,f)})\to {\mathbf X}$ is the maximal integral solution of (\ref{07042015-1828}) with $u(0)=\bar u$.

By Remark \ref{24022016-1242}, ${\mathbf \Phi}^{(p,f)}$ is a local semiflow on ${\mathbf X}$.
\begin{theorem}\label{15052015-1715}
If $u\in C (\R, {\mathbf X} )$ is a bounded solution of {\em (\ref{07042015-1828})}, then $\alpha(u)$  and $\omega(u)$ are nonempty, connected and compact in the space $C_0(0,l)$, ${\mathbf \Phi}_{t}^{(p,f)}(\alpha(u))=\alpha(u)$
and ${\mathbf \Phi}_{t}^{(p,f)}(\omega(u))=\omega(u)$ for all $t\geq 0$ and
$$
\alpha(u) \cup \omega(u)\subset {\cal E}
$$
where ${\cal E}$ is the set of all stationary solutions of {\em (\ref{07042015-1828})}.
\end{theorem}

\begin{lemma}\label{09022016-1850}
If $u\in C(\R, {\mathbf X})$ is a bounded solution of {\em (\ref{07042015-1828})}, then
$$
\sup_{t\in\R} \|u(t)\|_{W_0^{1,p}}< +\infty.
$$
\end{lemma}
\begin{proof} Suppose  the contrary. There exists $(t_n)$ in $\R$ such that $\|u(t_n)\|_{W_{0}^{1,p}} \to +\infty$ as $n\to +\infty$.
Clearly, due to Theorem \ref{06052015-1149} and the boundedness of $u(\R)$ in $C_0(0,l)$, we see that  $\varphi_{p,f}(u(t_n))\to +\infty$ and $t_n\to -\infty$, which simply implies
$$
\lim_{t\to -\infty} \varphi_{p,f}(u(t))=+\infty.
$$
On the other hand, for a.e. $t\in\R$,
\begin{align*}
\frac{\d}{\d t}\left(\frac{1}{2}\|u(t)\|_{L^2}^{2} \right) & = (\dot u(t), u(t))_{L^2}\\
& = - p \,\varphi_{p,f}(u(t))- p\int_{0}^{l} {\cal F}(x,u(t)(x)) \d x  + \int_{0}^{l} f(x,u(t)(x))u(t)(x)\d x\\
& \leq -p \,\varphi_{p,f}(u(t)) + C
\end{align*}
for some constant $C>0$. Hence there exists $t_0\in\R$ such that, for all $t\leq t_0$,
$$
\frac{\d}{\d t}\left(\frac{1}{2}\|u(t)\|_{L^2}^{2} \right) \leq - 1,
$$
which implies
$$
\frac{1}{2} \|u(t)\|_{L^2}^{2} \geq t_0-t + \frac{1}{2}\|u(t_0)\|_{L^2}^{2} \to + \infty, \ \mbox{ as } \to t \to - \infty,
$$
a contradiction. \end{proof}

\begin{proof}[Proof of Theorem \ref{15052015-1715}]
First observe that $u(\R)$ is a relatively compact subset of $C_0(0,l)$. Indeed, take any $(t_n)$ in $\R$. Then $u(t_n)={\mathbf \Phi}_1^{(p,f)}(\bar u_n)$ where $\bar u_n = u(t_n-1)$ for $n\geq 1$. The sequence $(\bar u_n)$ is bounded, therefore, due to Theorem \ref{1502-29072015}, $\left( {\mathbf
\Phi}_1^{(p,f)} (\bar u_n)\right)$ contains a convergent subsequence. Hence $\alpha (u)$ and $\omega (u)$ are nonempty and compact.
The ${\mathbf \Phi}^{(p,f)}$-invariance follows in a similar manner by the use of the compactness of $\Phi_1^{(p,f)}$. Furthermore, it follows from Theorem \ref{06052015-1149} that $u\in C(\R, W_{0}^{1,p}(0,l))$ and ${\varphi}_{p,f}\circ u$ is non-increasing, which means that the limits $\lim_{t \to -\infty} \varphi_{p,f} (u(t))$ and $\lim_{t\to +\infty} {\mathbf \varphi}_{p,f} (u(t))$ exist and the latter is finite.\\
\indent Now take $\bar u\in \alpha(u) \cup \omega(u)$ and suppose that either $t_n\to +\infty$ or $t_n\to -\infty$ and $u(t_n)\to \bar u$ as $n\to +\infty$. By the relative compactness of $u(\R)$,
for a fixed $\tau>0$, passing to a subsequence we may assume that
$u(t_n-\tau) \to \bar u_0$ in $C_0(0,l)$ for some $\bar u_0\in C_0(0,l)$ such that $\bar u = {\mathbf \Phi}_{\tau}^{(p,f)}(\bar u_0)$.
Since $u(\R)$ is bounded in $C_0(0,l)$ we see that, by Lemma \ref{09022016-1850}, that  $u(\R)$ is bounded in $W_{0}^{1,p}(0,l)$.
Therefore, by use of Theorem \ref{1502-29072015} (i), we have  $u(t_n) = {\mathbf  \Phi}_{\tau}^{(p,f)} (u(t_n-\tau)) \to {\mathbf \Phi}_{\tau}^{(p,f)} (\bar u_0)=\bar u$ in $W_0^{1,p}(0,l)$. For the same reason, we have
$u(t+t_n)={\mathbf \Phi}_{t}^{(p,f)} (u(t_n)) \to {\mathbf \Phi}_{t}^{(p,f)} (\bar u)$ in $W_{0}^{1,p}(0,l)$ for any $t> 0$.
Hence, for any $t\geq 0$,
$$
{\mathbf \varphi}_{p,f} ({\mathbf \Phi}_{t}^{(p,f)}(\bar u)) = \lim_{n\to +\infty} {\mathbf \varphi}_{p,f} (u(t+t_n)) = \lim_{n\to +\infty} {\mathbf \varphi}_{p,f} (u(t_n)) = {\mathbf \varphi}_{p,f} (\bar u).
$$
This  together with (\ref{15052015-1726}) means that ${\mathbf \Phi}_{t}^{(p,f)} (\bar u) = \bar u$ for all $t\geq 0$,
i.e. $\bar u\in {\cal E}$.
\end{proof}
\begin{proposition}\label{1650-20082015}
Any bounded set ${\mathbf N}\subset {\mathbf X}$ is strongly ${\mathbf \Phi}^{(p,f)}$-admissible {\em(}in the sense of Definition {\em \ref{1403-29072015}}{\em)}.
\end{proposition}
\begin{proof} Fix
$(t_n)$ in $(0,+\infty)$ and $(\bar u_n)$ in ${\mathbf N}$ such that
${\mathbf \Phi}^{(p,f)} ([0,t_n]\times \{ \bar u_n\})\subset {\mathbf N}$. Then observe that ${\mathbf \Phi}_{t_n}^{(p,f)}(\bar u_n) = {\mathbf \Phi}_{\tau}^{(p,f)} ({\mathbf \Phi}_{t_n-\tau}^{(p,f)}(\bar u_n)) \subset
{\mathbf \Phi}_{\tau}^{(p,f)} ({\mathbf N})$  for a fixed $\tau>0$ and  all sufficiently large $n\geq 1$. Hence, due to Theorem \ref{1502-29072015} (ii), $\left({\mathbf \Phi}_{t_n}^{(p,f)}(\bar u_n)\right)$ contains a convergent subsequence. Since ${\mathbf N}$ is bounded, the local semiflow ${\mathbf \Phi}^{(p,f)}$
does not blow up in ${\mathbf N}$, due to Theorem \ref{29072015-1333}, which completes the proof.
\end{proof}

\section{Continuity and compactness along $p$}

We start with a fundamental theorem on continuity and compactness of semigroups with respect to $p$.
\begin{theorem}\label{27072015-1215}$\mbox{ }$\\
{\em (i)} \parbox[t]{144mm}{If $p_n\geq 2$, $n\geq 1$, and $p_n\to p$ as $n\to +\infty$, then ${\mathbf A}_{p_n}\stackrel{\Gr}{\to}{\mathbf A}_p$  and therefore
$$
S_{{\mathbf A}_{p_n}} (t )\bar u    \to    S_{{\mathbf A}_p} (t)\bar u\  \text{ as } \ n\to +\infty\ \text{ for any }\bar u \in C_0(0,l)\text{ and }t\geq 0.
$$}\\[0.5em]
{\em (ii)} \parbox[t]{144mm}{For any bounded  $B\subset C_0 (0,l)$ and $t>0$, the set $\bigcup_{q \in [2,p]} S_{{\mathbf A}_q} (t) (B)$ is relatively compact in $C_0(0,l)$.}
\end{theorem}
\noindent In the proof we shall need the following convergence properties.
\begin{lemma}\label{1551-20082015}
If $p_n\geq 2$, $n\geq 1$, and $p_n\to p$ as $n\to +\infty$, then $\bar {\mathbf A}_{p_n}\stackrel{\Gr}{\to} \bar {\mathbf A}_p$
\end{lemma}
\begin{proof} Fix $u\in L^2(0,l)$ and $\lambda>0$. Put
\[
\ \ q:=\sup\{p_n\mid n\geq 1\}, \  \ \ q' := \frac{q}{q-1},\ \ \ v_n:=(I+\lambda\bar {\mathbf A}_{p_n})^{-1}u, \  \ \ z_n:=|v_n'|^{p_n-2}v_n'.
\]
Therefore, for any $n\geq 1$,
\begin{equation}\label{t3}\bar {\mathbf A}_{p_n}v_n=\frac 1\lambda(u-v_n) \text{ and hence }\int_0^l|v_n'|^{p_n} \d x=\frac 1\lambda\int_0^l (u-v_n)v_n \d x.\end{equation} From the contractiveness of the resolvents of $\bar {\mathbf A}_{p_n}$ it follows that $\|v_n\|_{L^2}\leq\|u\|_{L^2}$. Now, by means of the H\"older inequality, one may  show that
\begin{itemize}
\item $(v_n)$ is bounded in $H^1_0(0,l)$;
\item $(z_n)$ is bounded in $L^{q'}(0,l)$;
\item $(z'_n)$ is bounded in $L^2(0,l)$ (note that $z'_n=-\bar {\mathbf A}_{p_n}v_n$).
\end{itemize}
As a consequence, passing to a subsequence if necessary, we can assume that
\[v_n\to v\text{ in }L^2(0,l),\ v_n'\rightharpoonup v'\text{ in }L^2(0,l),\ z_n\to z\text{ in }L^{q'}(0,l),\ z_n'\rightharpoonup z'\text{ in }L^2(0,l)
\] and that $z_n\to z$ almost everywhere.
This clearly forces $v_n'\to |z|^{1/(p-1)}\mathop{\mathrm{sgn}}z$ almost everywhere.
Therefore $z=|v'|^{p-2}v'$, $z'=-\bar {\mathbf A}_p v$ and, clearly,
\[
 \bar {\mathbf A}_{p_n}v_n=\frac 1{\lambda}(u-v_n)\to \frac 1{\lambda}(u-v) \text{ in }L^2(0,l).
\]
On the other other hand
\[
\bar {\mathbf A}_{p_n}v_n = - z'_n \rightharpoonup - z' = \bar {\mathbf A}_{p}v
\]
which gives $v=\left(I+\lambda\bar {\mathbf A}_p\right)^{-1}u$ and ends the proof.
\end{proof}
\begin{proof}[Proof of Theorem \ref{27072015-1215}]
(i) Let $\bar u\in C_0 (0,l)$ and put $\bar u_n:=\left(I+\lambda {\mathbf A}_{p_n}\right)^{-1}\bar u$, $\bar u_0:=\left(I+\lambda {\mathbf A}_{p}\right)^{-1}\bar u$. Lemma \ref{1551-20082015} implies that $\bar u_n\to\bar u_0$ in $L^2(0,l)$.
By (\ref{15042015-1337}) we get for some constant $C>0$ that
\begin{align}
C\|\bar u_n-\bar u_0\|_{H^1_0}^{p_n} &\leq \left(\bar {\mathbf A}_{p_n}\bar u_n-\bar {\mathbf A}_p\bar u_0,\bar u_n-\bar u_0\right)_{L^2}-
\left(\bar {\mathbf A}_{p_n}\bar u_0-\bar {\mathbf A}_p\bar u_0,\bar u_n-\bar u_0\right)_{L^2}\nonumber \\
&\leq \|\lambda^{-1}(\bar u-\bar u_n)-\bar {\mathbf A}_p\bar u_0\|_{L^2}\cdot\|\bar u_n-\bar u_0\|_{L^2}+\label{eqM1}\\
&+\left\||\bar u_0'|^{p_n-2}\bar u_0'-|\bar u_0'|^{p-2}\bar u_0'\right\|_{L^2}\cdot\|\bar u_n'-\bar u_0'\|_{L^2}.\nonumber
\end{align}
The first term of the right-hand side of \eqref{eqM1} converges to zero.
Since $\bar u_0'\in C(0,l)$, the sequence $|\bar u_0'|^{p_n-2}\bar u_0'-|\bar u_0'|^{p-2}\bar u_0'$ is bounded in $C(0,l)$ and its pointwise limit equals zero. Therefore $\||\bar u_0'|^{p_n-2}\bar u_0'-\|\bar u_0'\|^{p-2}\bar u_0'\|_{L^2}\to 0$.  Observe, that in order to prove that the second term in \eqref{eqM1} tends to zero and therefore that $\bar u_n \to \bar u_0 $ in  $H_{0}^{1}(0,l)$, and consequently in $C_0 (0,l)$, it suffices to verify the boundedness of $\|\bar u_n'\|_{L^2}$.\\
\indent Put $\bar w_n=|\bar u_n'|^{p_n-2}\bar u_n'$. Then the sequence $\bar w_n'=-\mathbf{A}_{p_n}\bar u_n=\lambda^{-1}(\bar u_n-\bar u)$ is bounded in $L^{1}(0,l)$.  As $\bar u_n\in C_0(0,l)$, there exist points $x_n\in (0,l)$ such that $\bar u_n'(x_n)=0$ and therefore that $\bar w_n(x_n)=0$. Hence, $|\bar w_n(x)|\leq M$ for some $M>0$ and all $x\in[0,T]$. We finally obtain $|\bar u_n'(x)|\leq M^{1/(p_n-1)}$.\\
\indent (ii) We shall prove the assertion in two steps.\\
\noindent {\bf Step  1}. First we prove that for any bounded $B\subset L^2(0,l)$, $p>2$  and $t>0$ the set
$$
\bigcup_{q\in [2,p]} S_{\bar {\mathbf A}_q} (t) (B)
$$
is relatively compact. In view of Proposition \ref{23042015-2220}, it is sufficient to prove that
$$
L_\lambda:=\left\{u\in L^2(0,l)\mid u\in D(\varphi_q) \mbox{ and  } \varphi_q(u)\leq \lambda\text{ for some } 2\leq q\leq p\right\}
$$ is relatively compact for any $\lambda\geq 0$. To this end take any $u\in L_\lma$, i.e.  $\varphi_q (u) \leq \lambda$ for some $q\in [2,p]$. Then $u\in W^{1,q}_0(0,l)$ and $\int_0^l|u'|^q\leq p\lambda$. Therefore $u\in W^{1,2}_0(0,l)$ and
$$
\int_0^l |u'|^2\leq \left(\int_0^l |u'|^q\right)^{\frac 2q}\cdot l^{1-\frac 2q}\leq
(p\lambda)^{\frac 2q}\cdot l^{1-\frac 2q}.
$$
This shows that $L_\lambda$ is bounded in $H^{1}_0 (0,l)$, which due to the Rellich-Kondrachov compact embedding theorem means that  $L_\lma$ is relatively compact in $L^2(0,l)$, which ends the proof of Step 1.\\

\noindent {\bf Step 2}. Now let us take a bounded $B\subset C_0(0,l)$ and $t>0$. Take any $(p_n)$ in $[2,p]$ and $(\bar u_n)$ in $B$. Put $\alpha=t/3$ and define $u_n := S_{{\mathbf A}_{p_n}}(\cdot) \bar u_n$. We are going  to show that the sequence $(u_n(t))$ has a convergent subsequence in $C_0(0,l)$.\\
\indent By use of Step 1, without loss of generality we may assume that the sequence $(\tilde v_n):=(u_n(\alpha))$ converges in $L^2 (0,l)$ to some $\tilde v_0 \in L^2 (0,l)$ and $p_n\to p_0$. Put
\[\bar v_n := u_n(2\alpha)=S_{{\mathbf A}_{p_n}} (\alpha)\tilde v_n\text{ and }\bar v_0:=S_{\bar {\mathbf A}_{p_0}}(\alpha)\tilde v_0 .\] By Lemma \ref{1551-20082015}  and Proposition \ref{22042015-1020} (i), $\bar v_n \to \bar v_0$ in $L^2(0,l)$. Moreover, $\bar v_0\in W^{1,p_0}_0(0,l)$ in view of Proposition \ref{12042015-1740}, which implies that $\bar v_0\in C_0 (0,l)$.
Therefore we can put
\[v_0:= S_{{\mathbf A}_{p_0}}(\cdot)\bar v_0\colon [0,\alpha]\to C_0 (0,l),
\ v_n:=S_{{\mathbf A}_{p_n}} (\cdot)\bar v_n\colon [0,\alpha]\to C_0 (0,l).\]
Then, clearly $u_n(t)=v_n(\alpha)$ for all $n\geq 1$.  Further, we note that
\begin{equation}\label{24022016-1300}
\| v_n (\alpha)-v_0 (\alpha) \|_\infty \leq \|S_{\bar {\mathbf A}_{p_n}} (\alpha)\bar v_n - S_{\bar {\mathbf A}_{p_n}}(\alpha) \bar v_0\|_\infty +
\| S_{{\mathbf A}_{p_n}}(\alpha) \bar v_0 - S_{{\mathbf A}_{p_0}} (\alpha) \bar v_0\|_\infty
\end{equation}
By Lemma \ref{08042015-1551} (iii) and due to the continuity of the embedding of $H_{0}^{1}(0,l)$ into $C_0(0,l)$ we get
\[
\|S_{{\mathbf A}_{p_n}} (\alpha)\bar v_n - S_{{\mathbf A}_{p_n}}(\alpha) \bar v_0\|_\infty \leq
C\|S_{{\bar {\mathbf A}}_{p_n}} (\alpha)\bar v_n - S_{{\bar {\mathbf A}}_{p_n}}(\alpha) \bar v_0\|_{H_{0}^{1}} \leq C'
\|\bar v_n-\bar v_0\|_{L^2}^{1/p_n}\to 0
\]
for some constants $C,C'>0$.
Finally by Proposition \ref{22042015-1020} (i), we get $S_{{\mathbf A}_{p_n}}(\alpha) \bar v_0 \to  S_{{\mathbf A}_{p_0}} (\alpha) \bar v_0$ in $C_0 (0,l)$. This together with (\ref{24022016-1300}) proves that $u_n(t)=v_n(\alpha)\to v_0(\alpha)$ in $C_0 (0,l)$, which completes the proof.
\end{proof}

We will summarize continuity and compactness properties for the equation (\ref{07042015-1828}) in the following
\begin{theorem}\label{11052015-1524} $\mbox{ }$\\
{\em (i)} \parbox[t]{140mm}{{\em (Continuity)} Let $p_n\to p_0$ in $[2,+\infty)$ and $\bar u_n \to \bar u_0$ in $C_0 (0,l)$. Assume that the functions $f_n\colon [0,l] \times \R\to\R$, $n\geq 0$ satisfy \eqref{0128-29042015} with common Lipschitz constants and that $f_n(x,u)\to f_0(x,u)$ uniformly for $(x,u)$ from bounded subsets of $[0,l] \times \R$.
Then $\liminf_{n\to+\infty} T_{\bar u_n}^{(p_n, f_n)} \geq T_{\bar u_0}^{(p_0,f_0)}$. Moreover, if $u_n\colon [0,T] \to C_0(0,l)$, $n\geq 0$, with fixed $T>0$ are the solutions of {\em (\ref{07042015-1828})} with
$p=p_n$, $f=f_n$ $u_n(0)=\bar u_n$, and there is $R>0$ such that $\|u_n(t)\| \leq R$ for all $t\in [0,T]$, $n\geq 1$, then
$u_n \to u_0$ in $C([0,T], C_0(0,l))$.}\\[1em]
{\em (ii)} \parbox[t]{140mm}{{\em (Compactness)} Assume that $\{ p_n\}_{n\geq 1}$ is bounded in $[2,+\infty)$, $f_n\colon[0,l]\times\R\to\R$ are locally Lipschitz and such that the set $\bigcup_{n\geq 1} f_n([0,l]\times[-r,r])$ is bounded for any $r>0$. If $u_n\in C([0,T], C_0(0,l))$, $n\geq 1$, are solutions of {\em (\ref{07042015-1828})} with $p=p_n$, $f=f_n$  and $\bar u =\bar u_n$ and there is $R>0$ such that $\| u_n (t)\|_\infty \leq R$ for all $t\in [0,T]$ and $n\geq 1$, then for any $t\in (0,T]$ the sequence $(u_n(t))$ contains a convergent subsequence in the space $C_0 (0,l)$.}
\end{theorem}
\begin{proof} (i) follows directly from Theorem \ref{27072015-1215} and Proposition \ref{0208-29042015} (ii).
The condition (ii) comes from Theorem \ref{27072015-1215} and Proposition \ref{0208-29042015} (iii).
\end{proof}

At the end of this section, we express the obtained results in terms of parameterized semiflows.
To this end, let ${\mathbf X}=C_0(0,l)$ and let $f\colon[0,l]\times\R\times [0,1]\to\R$ be continuous with $f(x,0,\mu)=0$, for all $x\in [0,l]$ and $\mu\in [0,1]$, and such that for any $R>0$ there exists $L>0$ such that $|f(x,u,\mu)-f(x,v,\mu)|\leq L|u-v|$ whenever $x\in[0,l]$, $|u|,|v|\leq R$ and $\mu\in [0,1]$. Consider
\begin{equation}\label{1006-29042015}
\dot u(t) = - {\mathbf A}_{p} u(t) + {\mathbf F}(u(t),\mu), \ t>0,
\end{equation}
where ${\mathbf F}\colon {\mathbf X} \times [0,1]\to {\mathbf X}$ is defined by
\[[{\mathbf F}(u,\mu)](x) = f (x,u(x),\mu),\ x\in [0,l],\ \mu\in [0,1],\ u\in {\mathbf X}.\]
Define ${\mathbf \Phi}^{(p,\mu)}\colon  {\mathbf D}^{(p,\mu)}\to {\mathbf X}$  by
$$
{\mathbf D}^{(p,\mu)}:= \{ (t,\bar u)\in [0,+\infty)\times {\mathbf X} \mid t< T_{\bar u}^{(p,\mu)} \}
$$
and ${\mathbf \Phi}^{(p,\mu)} (t,\bar u) = {\mathbf \Phi}_t^{(p,\mu)} (\bar u):=u(t)$ where $u\colon [0, T_{\bar u}^{(p,\mu)})\to {\mathbf X}$ is the maximal integral solution of
(\ref{1006-29042015}) with $u(0)=\bar u$. The above results imply that this family of semiflows is continuous with respect to  $p$ and $\mu$ and that bounded sets are strongly admissible.
\begin{proposition}\label{0212-29042015} $\mbox{ }$\\
{\em (i)} \parbox[t]{140mm}{If $\bar u_n \to \bar u_0$ in $C_0(0,l)$, $t_n\to t_0$ in $[0,+\infty)$, $p_n\to p_0$ in $[2,+\infty)$, $\mu_n \to \mu_0$ in $[0,1]$ and $(t_0, \bar u_0)\in {\mathbf D}^{(p_0, \mu_0)}$, then $(t_n, \bar u_n)\in {\mathbf D}^{(p_n, \mu_n)}$ for large $n$ and ${\mathbf \Phi}_{t_n}^{(p_n, \mu_n)}(\bar u_n) \to {\mathbf \Phi}_{t_0}^{(p_0, \mu_0)}(\bar u_0) $ as $n\to +\infty$.}\\[1em]
{\em (ii)} \parbox[t]{140mm}{If ${\mathbf N}\subset {\mathbf X}$ is bounded, then, for any $(p_n )$ in $ [ 2, +\infty )$ an $(\mu_n)$ in $[0,1]$, ${\mathbf N}$ is strongly $({\mathbf \Phi}^{(p_n, \mu_n)})$-admissible.}
\end{proposition}
\begin{proof} (i) is a straightforward consequence of  Theorem  \ref{11052015-1524} (i).
The proof (ii) goes along the lines of Proposition \ref{1650-20082015} with use of Theorem \ref{11052015-1524} (ii).
\end{proof}

\section{Proof of Theorem \ref{0914-29032015}}

We start with a computation of Conley index at zero and at infinity.
\begin{theorem} \label{11052015-1559}
Under the assumptions of Theorem {\em \ref{0914-29032015}}\\
{\em (i)} \parbox[t]{140mm}{$K_0:=\{0\}$ is an isolated invariant set with respect to ${\mathbf \Phi}^{(p,f)}$ and
$$
h({\mathbf \Phi}^{(p,f)},K_0 )=\Sigma^{k_0},
$$
where $k_0$ is such that $\lma_{k_0} \leq f'_0(x) \leq \lma_{k_0+1}$, for all $x\in [0,l]$, and the inequalities are strict on a set of positive measure.}\\[1em]
{\em (ii)} \parbox[t]{140mm}{$K_\infty$ consisting of all bounded full orbits for ${\mathbf \Phi}^{(p,f)}$ is a bounded isolated invariant set with a strongly admissible neighborhood and
$$
h({\mathbf \Phi}^{(p,f)}, K_\infty) = \Sigma^{k_\infty},
$$
where $k_\infty$ is such that $\lma_{k_\infty}\leq f'_\infty(x) \leq \lma_{k_\infty+1}$, for all $x\in [0,l]$, and the inequalities are strict on a set of positive measure.}
\end{theorem}
The computation of the indices of $K_0$ and $K_\infty$ will be reduced to computing the Conley index of zero in a special case.
\begin{lemma}\label{15052015-0609}
Let $g\in C([0,l])$ be a such that for some $k\geq 1$, $\lma_k^{(p)} \leq g(x) \leq \lma_{k+1}^{(p)}$ for all $x\in [0,l]$ with the strict inequalities on a set of positive measure and let $\{{\mathbf \Phi}^{(p,g)} \}_{t\geq 0}$ be the local semiflow generated on $C_0(0,l)$ by the problem
\begin{equation}\label{11052015-1622}
\left\{ \begin{array}{l}
u_t =  \left(\left| u_x \right|^{p-2}u_x \right)_x + g(x)|u|^{p-2}u, \ x\in (0,l), \ t>0,\\
u(t,0)=u(t,l)=0, \ t>0.
\end{array}
\right.
\end{equation}
Then  $u \equiv 0$ is the only full bounded solution of ${\mathbf \Phi}^{(p,g)}$ and, in particular, $K:=\{ 0\}$ is an isolated invariant set relative to ${\mathbf \Phi}^{(p,g)}$.
Moreover $h({\mathbf \Phi}^{(p,g)}, K)=\Sigma^k$.
\end{lemma}
\noindent In the proof we shall use the following lemma.
\begin{lemma}\label{1719-20082015}{\em (See e.g. \cite{Del-Man-El} and \cite{Manasevich})}
If $g$ is as in Lemma \ref{15052015-0609}, then  the problem
\begin{equation}
\left\{ \begin{array}{l}
-  \left(\left| u_x \right|^{p-2}u_x \right)_x = g(x)|u|^{p-2}u, \ x\in (0,l),\\
u(0)=u(l)=0.
\end{array}
\right.
\end{equation}
has no nontrivial weak solutions.
\end{lemma}
\begin{proof}[Proof of Lemma \ref{15052015-0609}]
Given a full bounded solution $u\in C(\R,C_0(0,l))$ of (\ref{11052015-1622}). By Theorem \ref{15052015-1715},
$\alpha(u)\cup \omega(u)\subset {\cal E}$. If  $u$ was nontrivial, then we would get
two different equilibria, since due to Theorem \ref{06052015-1149}  the Lyapunov function would change its value along the nontrivial solution. But according to Lemma \ref{1719-20082015}, in this case we have ${\cal E}=\{0\}$, which is a contradiction showing that there are no nontrivial bounded solutions of (\ref{11052015-1622}), which shows that
$K=\{0\}$ is an isolated invariant set.\\
 \indent Observe that we get
 \begin{equation}\label{2221-20082015}
 h({\mathbf \Phi}^{(p,g)}, K) = h({\mathbf \Phi}^{(p,\lma)}, K) \ \ \mbox{  for any } \ \ \lma \in (\lma_k^{(p)}, \lma_{k+1}^{(p)}).
 \end{equation}
Indeed, it is enough to consider a family of semiflows ${\mathbf \Phi}^{(p,\tilde g(\cdot, \mu))}$, $\mu\in [0,1]$, with
$\tilde g\colon  [0,l]\times [0,1] \to [\lma_{k}^{(p)}, \lma_{k+1}^{(p)}]$ given by
$\tilde g(x,\mu)=\mu g(x)+(1-\mu)\lma$, $x\in [0,l]$, $\mu\in [0,1]$, and use Proposition \ref{0212-29042015} and the first part of the proof together with the continuation property $(H4)$ of Conley index.\\
\indent  Next we consider the family $\{ \Phi^{(\tilde p(\mu), \tilde \lma (\mu))}\}_{\mu\in [0,1]}$ of local semiflows on $C_0(0,l)$ with  continuous functions $\tilde p\colon  [0,1]\to [2,p]$ and $\tilde \lma \colon [0,1] \to \R$ with  $\tilde p (0) = p$, $\tilde p (1)=2$  and $\lma_k^{\tilde p(\mu)}<\tilde \lma (\mu) <\lma_{k+1}^{\tilde p(\mu)}$ for all $\mu\in [0,1]$. Again, in view of Proposition \ref{0212-29042015} we may apply the continuation property  of Conley index and obtain
 \begin{equation}\label{2235-20082015}
 h({\mathbf \Phi}^{(p, \tilde \lma (0))}, K)=h({\mathbf \Phi}^{(2,\tilde \lma (1))}, K).
 \end{equation}
Using the spectral decomposition given by the Laplace operator ${\mathbf A}_2$ (and $\bar{\mathbf A}_2$) together with Theorem \ref{11052015-1744},  we get
\begin{equation}\label{2218-20082015}
h({\mathbf \Phi}^{(2, \tilde \lma(1))}, K) = \Sigma^k.
\end{equation}
By combining (\ref{2221-20082015}), (\ref{2235-20082015}) and (\ref{2218-20082015}) we get the assertion.
\end{proof}

\begin{lemma}\label{L:conv1} Under the assumptions \eqref{1219-21082015} and \eqref{1220-21082015}, if $\R\ni r_n\to 0^+$, $\R\ni R_n\to\infty$ and $\bar v_n\to \bar v_0$ in $C_0(0,l)$, then
\[r_n^{1-p}{\mathbf F}(r_n\bar v_n)\to f'_0|\bar v_0|^{p-2}\bar v_0\text{ and }R_n^{1-p}{\mathbf F} (R_n\bar v_n)\to f'_\infty|\bar v_0|^{p-2}\bar v_0\text{ in }C_0(0,l).\]
\end{lemma}
\begin{proof}The first convergence follows directly from \eqref{1219-21082015}. We shall prove the second convergence.
Fix $\varepsilon>0$. Put $V:=\sup\{\|\bar v_n\|^{p-1}\mid n\in\N\}$ and let $D>0$ be such that
\begin{equation}\label{c1}
\left|\frac{f(x,u)}{|u|^{p-2}u}-f_\infty'(x)\right|<\frac\varepsilon {2V}\text{ for all } x\in [0,l] \mbox{ and } |u|\geq D.
\end{equation}
For $n$ sufficiently large
\begin{equation}\label{c2}
\left| f'_\infty(x)|\bar v_n(x)|^{p-2}\bar v_n(x) - f'_\infty(x)|\bar v_0(x)|^{p-2}\bar v_0(x) \right|<\frac\varepsilon 2\text{ for all }x\in [0,l]
\end{equation}
and
\begin{equation}\label{c3}\frac{\sup|f([0,l]\times[-D,D])|+\|f'_\infty\|D^{p-1}}{R_n^{p-1}}<\frac\varepsilon 2.
\end{equation}
Then
\begin{equation}\label{c4}
\left|\frac{f(x,R_n\bar v_n(x))}{R_n^{p-1}}-f'_\infty(x)|\bar v_0(x)|^{p-2}\bar v_0(x) \right|<\varepsilon\text{ for }x\in[0,l].
\end{equation}
Indeed, set $x\in [0,l]$. If $|R_n\bar v_n(x)|\geq D$, then \eqref{c4} follows from \eqref{c1} and \eqref{c2}.
If $|R_n\bar v_n(x)|< D$, then \eqref{c4} is implied by \eqref{c2} and \eqref{c3}.
\end{proof}
\begin{proof}[Proof of Theorem \ref{11052015-1559}] (i) Consider $\tilde f(x,u,\mu):= \mu f(x,u) + (1-\mu) f'_0 (x) |u|^{p-2}u$, for $x\in [0,l]$, $u\in\R$, $\mu\in [0,1]$, and the family of local semiflows $\{\tilde{\mathbf \Phi}^{(\mu)} \}_{\mu\in [0,1]}$ on ${\mathbf X}=C_0(0,l)$ generated by the equations
\begin{equation}\label{15052015-0618}
\left\{ \begin{array}{l}
u_t =  \left(\left| u_x \right|^{p-2}u_x \right)_x + \tilde f(x,u,\mu), \ x\in (0,l), \ t\in\R,\\
u(t,0)=u(t,l)=0,\ t\in \R.
\end{array}
\right.
\end{equation}
 We claim that there exists $r_0>0$ such that for all $r\in (0,r_0]$, ${\mathbf N}_r:=B(0,r)$ is an isolating neighborhood of $K_0$ relative to $\tilde {\mathbf \Phi}^{(\mu)}$ for all $\mu\in [0,1]$. Suppose to the contrary, that there exist $(r_n)$ in $(0,+\infty)$, $(\mu_n)$ in $[0,1]$ with $r_n\to 0^+$ together with full solutions $u_n\in C(\R,{\mathbf X})$ of $\tilde {\mathbf \Phi}^{(
 \mu_n)}$ such that $\|u_n(0)\|_\infty = \sup_{t\in\R} \|u_n(t)\|_\infty=r_n$ for all $n\geq 1$. Define $v_n\in C(\R,{\mathbf X})$ by $v_n(t):=r_n^{-1}  u_n (t/r_n^{p-2})$, $t\in\R$, $n\geq 1$. Then $\| v_n(0) \|_\infty=1$ and $v_n$ is a solution of
$$
\left\{ \begin{array}{l}
u_t =  \left(\left| u_x \right|^{p-2}u_x \right)_x + f_n(x,u), \ x\in (0,l), \ t\in\R,\\
u(t,0)=u(t,l)=0, t\in \R,
\end{array}
\right.
$$
where $f_n\colon [0,l]\times \R\to\R$, $f_n (x,u):= r_n^{-(p-1)} \tilde f (x, r_n u,\mu_n)$, $x\in [0,l]$, $u\in \R$, $n\geq 1$.
By \eqref{1219-21082015} and \eqref{1220-21082015}, $|f(x,u)|\leq M|u|^{p-1}$ for some $M>0$. Therefore, $r^{1-p}|f(x,ru)|\leq M|u|^{p-1}$ and consequently
\begin{equation}\label{m1}\|{\mathbf F}_n(v_n(t))\|_\infty \leq M \text{ for }t\in\R,
\end{equation}
where $[{\mathbf F}_n(\bar u)](x)=f_n(x,\bar u(x))$  for $\bar u\in C_0(0,l),\ x\in[0,l]$.
Thus, we can use Proposition \ref{22042015-1020} (ii), which together Theorem \ref{1502-29072015} and a diagonal argument gives a subsequence  (still denoted by $(v_n)$)  converging pointwise  to some $v_0\in C(\R, {\mathbf X})$ . From Lemma \ref{L:conv1} and \eqref{m1} it follows that ${\mathbf F}_n\circ v_n\to f'_0|v_0|^{p-2}v_0$ in $L^1([0,T],C_0(0,l))$. Therefore, Proposition \ref{22042015-1020} (i) gives that $v_0$ is a nonzero full bounded integral solution of
\[
\left\{ \begin{array}{l}
v_t =  \left(\left| v_x \right|^{p-2}v_x \right)_x + f'_0(x)|v|^{p-2}v, \ x\in (0,l), \ t\in\R,\\
v(t,0)=v(t,l)=0, t\in R.
\end{array}
\right.
\]
This, in view of Lemma \ref{15052015-0609}, gives a contradiction, which proves the existence of a proper isolating neighborhood ${\mathbf N}_r$ of $K_0$ relative to the local semiflows $\tilde {\mathbf \Phi}^{(\mu)}$, $\mu\in [0,1]$.
Now Proposition \ref{0212-29042015}, the homotopy invariance of Conley index (H4) and Lemma \ref{15052015-0609} imply the required equality
$$
h({\mathbf \Phi}^{(p,f)}, K_0)=h(\tilde {\mathbf \Phi}^{(1)}, K_0) = h(\tilde {\mathbf \Phi}^{(0)}, K_0) = \Sigma^{k_{0}}.
$$
\noindent (ii) The proof is analogous to that for (i). Here we change properly the definition of $\tilde f$: $\tilde f(x,u,\mu):= \mu f(x,u) + (1-\mu) f'_\infty (x) |u|^{p-2}u$, for $x\in [0,l]$, $u\in\R$, $\mu\in [0,1]$, and consider the family of local semiflows $\{\tilde{\mathbf \Phi}^{(\mu)} \}_{\mu\in [0,1]}$ on ${\mathbf X}=C_0(0,l)$ generated by the equations (\ref{15052015-0618}). We can show that there exists $R_0>0$ such that for all $R>R_0$ ${\mathbf N}_R:=B(0,R)$ is an isolating neighborhood of $K_\infty$ relative to $\tilde {\mathbf \Phi}^{(\mu)}$ for all $\mu\in [0,1]$. Indeed, on the contrary, suppose that there exist $R_n\to +\infty$  and $\mu_n \in [0,1]$, $n\geq 1$ together with full solutions $u_n\in C(\R,{\mathbf X})$ of $\tilde {\mathbf \Phi}^{(\mu_n)}$ such that $\|u_n(0)\|_\infty = \sup_{t\in\R} \|u_n(t)\|_\infty=R_n$. Define $v_n\in C(\R,{\mathbf X})$ by $v_n(t):=R_n^{-1}  u_n (t/R_n^{p-2})$, $t\in\R$, $n\geq 1$. Then $\| v_n(0) \|_\infty=1$ and $v_n$ is an integral solution of
$$
\left\{ \begin{array}{l}
u_t =  \left(\left| u_x \right|^{p-2}u_x \right)_x + f_n(x,u), \ x\in (0,l), \ t\in\R,\\
u(t,0)=u(t,l)=0, t\in \R,
\end{array}
\right.
$$
where $f_n\colon [0,l]\times \R\to\R$, $f_n (x,u):= R_{n}^{-(p-1)} \tilde f  (x, R_n u,\mu_n)$, $x\in [0,l]$, $u\in \R$, $n\geq 1$.
In a similar manner as in (i), there exists a subsequence of $(v_n)$ converging uniformly on bounded intervals to some bounded  $v_0\in C(\R,{\mathbf X})$ with $\|v_0(0)\|=1$ that is an integral solution of
$$
\left\{ \begin{array}{l}
v_t =  \left(\left| v_x \right|^{p-2}v_x \right)_x + f'_\infty (x)|v|^{p-2}v, \ x\in (0,l), \ t\in\R,\\
v(t,0)=v(t,l)=0, t\in \R,
\end{array}
\right.
$$
which is impossible. Finally, using the homotopy invariance of Conley index and Lemma \ref{15052015-0609} one has
$$
h({\mathbf \Phi}^{(p,f)}, K_\infty)=h(\tilde {\mathbf \Phi}^{(1)}, K_\infty) = h(\tilde {\mathbf \Phi}^{(0)}, \{ 0 \}) = \Sigma^{k_{\infty}},
$$
which completes the proof.
\end{proof}

\begin{proof}[Proof of Theorem \ref{0914-29032015}] By Theorem \ref{11052015-1559}, $(K_0, {\mathbf \Phi}^{(p,f)}), (K_\infty, {\mathbf \Phi}^{(p,f)})\in {\cal I}({\mathbf X})$ and
$$
h({\mathbf \Phi}^{(p,f)}, K_0)=\Sigma^{k_0}  \ \ \  \mbox{ and }  \ \ \ h({\mathbf \Phi}^{(p,f)}, K_\infty)=\Sigma^{k_\infty}.
$$
This means that $K_0$ and $K_\infty$ are irreducible and that $h({\mathbf \Phi}^{(p,f)}, K_0) \neq h({\mathbf \Phi}^{(p,f)}, K_\infty)$. Hence, in view of Theorem \ref{15052015-0642}, we get a full solution $u\colon \R\to {\mathbf X}$
of ${\mathbf \Phi}^{(p,f)}$ with $u(\R)\not\subset K_0$ and such that either $\alpha(u) \subset K_0$ or $\omega(u)\subset K_0$.
In view of Theorem \ref{15052015-1715}, since $u$ is a nontrivial solution of ${\mathbf \Phi}^{(p,f)}$, there is a nontrivial stationary solution $\bar u$ of  \eqref{07042015-1828} such that either $\lim_{t\to -\infty} u(t)=0$ and $\bar u \in \omega(u)$ or
$\bar u \in \alpha (u)$ and $\lim_{t\to + \infty} u(t)=0$. Since each solution of ${\mathbf \Phi}^{(p,f)}$ is a solution of (\ref{07042015-1828}), the proof is completed.
\end{proof}

\section*{Acknowledgements}
The study of both the authors was supported by the NCN Grant 2013/09/B/ST1/01963.

\end{document}